\documentclass{amsart}

\usepackage{amssymb,amsmath,amsfonts,amsthm,epsfig,amscd}
\usepackage{commath}
\usepackage{stmaryrd}
\usepackage[all,cmtip,poly]{xy}
\usepackage{svn}
\usepackage{mathtools}
\usepackage{mathrsfs}
\usepackage{tikz-cd}
\usepackage{hyperref}
\usepackage{nameref}
\usepackage{dsfont}
\usepackage{stackrel}

\newtheorem{lem}{Lemma}[section]

\newtheorem{df}[lem]{Definition}

\newtheorem{cor}[lem]{Corollary}
\newtheorem{thm}[lem]{Theorem}

\newtheorem{prop}[lem]{Proposition}

\theoremstyle{remark}
\newtheorem{rem}[lem]{Remark}

\newcommand{\GL}{\mathrm{GL}}
\newcommand{\sO}{\ensuremath{\mathscr{O}}}

\newcommand{\A}{\AA}

\newcommand{\Q}{\QQ}

\newcommand{\Z}{\ZZ}

\newcommand{\m}{\frakm}

\renewcommand{\AA}{{\mathbb A}}

\def\Ban{\mathrm{Ban}}
\def\ad{\mathrm{ad}}

\def\Ind{\mathrm{Ind}}
\def\Int{\mathrm{Int}}

\def\Hom{\mathrm{Hom}}

\def\for{\mathrm{for}}

\def\Ext{\mathrm{Ext}}
\def\Sp{\mathrm{Sp}}
\def\End{\mathrm{End}}
\def\Res{\mathrm{Res}}
\def\Rep{\mathrm{Rep}}

\def\contadg{\mathrm{Ban}^{\mathrm{ad}}_{G_0}(K)}

\def\reptopt{\mathrm{Rep}^{\mathrm{top.c}}_K(T^+)}

\def\replag{\mathrm{Rep}^{\mathrm{la}}_K(G)}

\newcommand{\wtG}{\widetilde{G}}
\newcommand{\wtI}{\widetilde{I}}

\newcommand{\wtK}{\widetilde{K}}
\newcommand{\wtm}{\widetilde{\frakm}}
\newcommand{\la}{\mathrm{la}}
\newcommand{\an}{\mathrm{an}}
\newcommand{\fs}{\mathrm{fs}}

\newcommand{\BB}{{\mathbb B}}

\newcommand{\FF}{{\mathbb F}}
\newcommand{\GG}{{\mathbb G}}
\newcommand{\HH}{{\mathbb H}}

\newcommand{\NN}{{\mathbb N}}

\newcommand{\PP}{{\mathbb P}}
\newcommand{\QQ}{{\mathbb Q}}

\newcommand{\TT}{{\mathbb T}}

\newcommand{\ZZ}{{\mathbb Z}}

\newcommand{\bC}{\ensuremath{\mathbf{C}}}
\newcommand{\bD}{\ensuremath{\mathbf{D}}}

\newcommand{\bK}{\ensuremath{\mathbf{K}}}
\newcommand{\bL}{\ensuremath{\mathbf{L}}}

\newcommand{\bT}{\ensuremath{\mathbf{T}}}

\newcommand{\bh}{\ensuremath{\mathbf{h}}}

\newcommand{\cA}{{\mathcal A}}

\newcommand{\cC}{{\mathcal C}}
\newcommand{\cD}{{\mathcal D}}
\newcommand{\cE}{{\mathcal E}}

\newcommand{\cH}{{\mathcal H}}
\newcommand{\cI}{{\mathcal I}}

\newcommand{\cL}{{\mathcal L}}

\newcommand{\cO}{{\mathcal O}}
\newcommand{\cP}{{\mathcal P}}

\newcommand{\cS}{{\mathcal S}}

\newcommand{\cV}{{\mathcal V}}

\newcommand{\frakm}{\mathfrak{m}}
\newcommand{\frakn}{\mathfrak{n}}

\newcommand{\Qbar}{\overline{\Q}}

\newcommand{\Qpbar}{\Qbar_p}

\makeatletter
\tikzset{
  column sep/.code=\def\pgfmatrixcolumnsep{\pgf@matrix@xscale*(#1)},
  row sep/.code   =\def\pgfmatrixrowsep{\pgf@matrix@yscale*(#1)},
  matrix xscale/.code=%
    \pgfmathsetmacro\pgf@matrix@xscale{\pgf@matrix@xscale*(#1)},
  matrix yscale/.code=%
    \pgfmathsetmacro\pgf@matrix@yscale{\pgf@matrix@yscale*(#1)},
  matrix scale/.style={/tikz/matrix xscale={#1},/tikz/matrix yscale={#1}}}
\def\pgf@matrix@xscale{1}
\def\pgf@matrix@yscale{1}
\makeatother 

\title{A derived construction of eigenvarieties}
\begin{document}
\author{Weibo Fu}
\email{wfu@math.princeton.edu}
\date{Oct,~2022}

\begin{abstract}
We construct a derived variant of Emerton's eigenvarieties using the locally analytic representation theory of $p$-adic groups.
The main innovations include comparison and exploitation of two homotopy equivalent completed complexes associated to the locally symmetric spaces of a quasi-split reductive group $\GG$, comparison to overconvergent cohomology, proving exactness of finite slope part functor, together with some representation-theoretic statements. 
As a global application, we exhibit an eigenvariety coming from data of $\GL_n$ over a CM field as a subeigenvariety for a quasi-split unitary group.
\end{abstract}
\maketitle
\tableofcontents

\section{Introduction}
Eigenvarieties have been playing important roles in the developments of the Langlands program, such as the developments of the Fontaine-Mazur conjecture (\cite{Boc01}, \cite{Kis03}, \cite{Eme11} for examples) and the symmetric power functoriality \cite{NT21a}, \cite{NT21b}, \cite{NT20J}, both for modular forms.
They were previously constructed for instances by \cite{AIP15}, \cite{AS08}, \cite{BHS16}, \cite{Buz07}, \cite{Eme06I}, \cite{Han17}, \cite{Urb11} using different approaches.

The goal of this article is to introduce a generalization of Emerton's construction of eigenvarieties for a quasi-split reductive group in his seminal paper \cite{Eme06I}.
For a general quasi-split reductive group over a number field, especially for groups with nonzero defects, we believe that our eigenvarieties are quite different than the ones constructed in \cite{Eme06I} due to existence of nonvanishing higher cohomology of a complex of locally analytic representations of a $p$-adic group and the corresponding higher locally analytic Jacquet modules.

If $\GG$ is a quasi-split reductive group over a number field $F$. 
We may form a complex to compute the completed cohomology of the locally symmetric spaces of $\GG$.
We have two ways of picking such a complex given the locally symmetric spaces (\S \ref{comp}):
one canonical way is to consider the set of all singular cochains of the spaces, and we get an enormous complex of Banach representation $A^\bullet$ of the group $\GG(F_p)$,
another way is to only use cochains given by a fixed choice of finite triangulation of the Borel-Serre compactification of $X_{\GG, K_pK^p}$ for a base level subgroup $K_p K^p \subset \GG(F_p) \times \GG(\AA_F^p)$, accordingly we get a complex $S^\bullet$ of admissible Banach representations of the compact group $K_p$.
For defining Jacquet modules associated to a tame level $K^p$, both ways alone seem to be concerning: an arbitrary Banach representation does not have good functional analysis properties, while the action of $K_p$ on $S^\bullet$ does not extend to $\GG(F_p)$ and the Hecke action for $S^\bullet$ is missing before passing to cohomology.

Using ingredients and notions in \S \ref{LR} and \S \ref{Exact FS}, we illustrate how to make use of both $A^\bullet$ and $S^\bullet$ via a homotopy equivalence $A^\bullet \leftrightarrows S^\bullet$ to construct derived Jacquet modules associated to the tame level $K^p$.
Eventually we construct eigenvarieties out of these essentially admissible representations.
The idea of using a finite triangulation leads to admissibility of completed cohomology by \cite{Eme06I} and constructions of eigenvarieties by \cite{AS08}, \cite{Urb11}, \cite{Han17}.
Hao Lee \cite{Lee22} studies a similar but different question, and he constructs the derived Jacquet-Emerton modules using a different approach.

We establish some technical results concerning locally analytic representations in \S \ref{LR} for uses in further sections.
In \S \ref{Exact FS}, we prove derived Jacquet modules associated to $A^\bullet \leftrightarrows S^\bullet$ are canonically defined independent of various choices and they are essentially admissible representations of the $p$-adic maximal torus of $\GG$ in Prop \ref{GAEA}. 
A key ingredient Theorem \ref{FSEX} in \S \ref{Exact FS} is to establish exactness of the finite slope part functor introduced in \cite{Eme06A}:
let $T_0$ be a compact $p$-adic torus, $Y^+$ be $T_0 \times \ZZ_{+}$ and $Y$ be $T_0 \times \ZZ$, we have
\begin{thm}
Let $0 \to U \to V \to W \to 0$ be a short exact sequence of certain $Y^+$ representations of compact type such that $U$ has subspace topology and $W$ has quotient topology.
Then the finite slope part of the sequence in the sense of \cite{Eme06A} $0 \to U_\fs \to V_\fs \to W_\fs \to 0$ is exact in the category of locally analytic representations of $Y$, where $U_\fs$ is a closed subspace and $W_\fs$ has quotient topology.
\end{thm}

If $I_p$ is an Iwahori level at $p$ with a maximal unipotent subgroup $N_p$, in \S \ref{comp}, we give a comparison of cohomology Theorem \ref{comp to overconvergent} of $X_{\GG, I_pK^p}$ with coefficient rising from a locally analytic representation \[ \Ind^{\la}_{I_p} := \{f \in C^{\la}(I_p,E) | f(gn)=f(g), ~ \mathrm{for} ~ \forall ~ n \in N_p\} \] of $I_p$ to the cohomology of $N_p$ invariants of locally analytic vectors of $S^\bullet$ (or $A^\bullet$).
\begin{thm}
There is a canonical isomorphism of cohomology  \[ H^\bullet(X_{\GG,I_p K^p}, \Ind^{\la}) \xrightarrow{\sim} H^\bullet \left( ((S^\bullet \otimes_{\cO_E} E)^{\la})^{N_p}\right). \]
\end{thm}
And a similar comparison holds for compactly supported cohomology.
These comparisons help us deduce vanishing results for applications.
For example, we are able to deduce vanishing of compactly supported cohomology of $N_p$ invariants of locally analytic vectors of $A^\bullet$ with the comparison using the main result of Caraiani-Scholze \cite{CS19}.

For our main global application in \S \ref{App uni}, let $\wtG$ be a quasi-split unitary group with Siegel parabolic $P$ whose Levi component is $G = \GL_n$ over a totally imaginary CM field.
Let $\m$ be a maximal ideal of derived Hecke algebra of $G$ with certain good tame level $K^p$, and $\wtm$ be the pullback of $\m$ via the Satake morphism of derived Hecke algebra of $\wtG$ with tame level $\wtK^p$ compatible with $K^p$.

 \[ \mathrm{Let} ~ \wtG_p := \wtG(F_p), ~ P_p := P(F_p), ~ G_p := G(F_p) \] \[ \wtG_0 := G(\cO_{F_p}), ~ P_0 := P(\cO_{F_p}), ~ G_0 := G(\cO_{F_p}). \]
In \S \ref{EV construction}, we will construct eigenvarieties $\mathscr{E}^\ast_{\wtG}(\wtK^p, \wtm)$ for the unitary group $\wtG$ and a suitable maximal ideal $\wtm$ and a suitable tame level $\wtK^p$.
For such data, the most interesting eigenvariety is $\mathscr{E}^d_{\wtG}(\wtK^p, \wtm)$, where degree $d$ is the middle dimension of symmetric space of $\wtG$.

Let $\widetilde{\pi}(G, K^p) \leftrightarrows \pi(G, K^p)$ be two complexes computing completed cohomology of $G$ and tame level $K^p$ as we introduced in the previous paragraph ($\pi(G, K^p)$ is an admissible Banach representation of $G_0$ and $\widetilde{\pi}(G, K^p)$ is a Banach representation of $G_p$). 
The homotopy equivalence induces the corresponding one on inductions $\Ind^{\wtG_p}_{P_p} \widetilde{\pi}(G, K^p) \leftrightarrows \Ind^{\wtG_0}_{P_0} \pi(G, K^p)$ by Prop \ref{Ind ext}.
We may replace the right hand side by an injective resolution of admissible Banach representations of $\wtG_0$, and apply our eigenvariety machine in \S \ref{EV construction} to obtain the induction eigenvarieties $\mathscr{E}^\ast(\Ind, G, K^p, \m)$ for a maximal ideal $\frakm$ in various degrees.
As we have seen, these eigenvarieties $\mathscr{E}^\ast(\Ind, G, K^p, \m)$ come from a parabolic induction of the completed cohomology complexes associated to $G$.
Our main theorem in this section (\S \ref{App uni}, Theorem \ref{main glo}) establishes a type of functoriality between 
$\mathscr{E}^d_{\wtG}(\wtK^p, \wtm)$ and $\mathscr{E}^d(\Ind, G, K^p, \m)$.

\begin{thm}
Suppose $F^+ \neq \QQ$, $d$ is the middle dimension of symmetric space of $\wtG$, and $\widetilde{\m}$ satisfies the condition of \cite[Thm 1.1]{CS19}.
The induction eigenvariety $\mathscr{E}^d(\Ind, G, K^p, \m)$ is embedded as a subvariety of the $\wtG$-eigenvariety $\mathscr{E}^d_{\wtG}(\wtK^p, \widetilde{\frakm})$.
\end{thm}

A highlight is that the eigenvariety $\mathscr{E}^d(\Ind, G, K^p, \m)$ only depends on $G$. 
Although there seems to be serious difficulties to figure out the exact relationship between such an eigenvariety and the $\GL_n$-eigenvariety at the moment.

\subsection*{Acknowledgments}
I would like to thank my advisor Richard Taylor for suggesting me to think about functoriality of eigenvarieties and checking my arguments with numerous discussions.
I would also like to thank Matthew Emerton and Hao Lee for helpful discussions, as well as the anonymous referee for detailed corrections and expository suggestions.

\section{Notation}\label{notation}

If $R$ is a (possibly non-commutative) ring, then we will write $\text{Mod}(R)$ for the abelian category of (left unital) $R$-modules and $\mathbf{D}(R)$ for the derived category of $R$-modules.
If $K$ is a field, we use $\text{Vect}_K$ to denote the abelian category of $K$-vector spaces.
If $\cA$ is an additive category, we will write $\bC\bh(\cA)$, $\bK(\cA)$ and $\bC\bh^b(\cA)$, $\bK^b(\cA)$ for the category of chain complexes, the homotopy categories and their full subcategory consisting of bounded complexes of $\cA$.
If $\cA$ is an abelian category, we will write $\bD(\cA)$ and $\bD^b(\cA)$ for the derived category and the bounded derived category of $\cA$.
Note that in \S \ref{comp}, we actually consider cochain complexes like $A^\bullet_{r,s}, S^\bullet_{r,s}$, we consider them as chain complexes up to the usual normalization. 

If $K$ is a local field, then we will write $\cO_K$ as its ring of integers with an uniformizer $\varpi_K$, and residue field $k_K$.
If $V$ and $W$ are two $LB$-spaces (\cite[Def 1.1.16]{Eme17}), the inductive tensor product topology and projective tensor product topology agree on $V \otimes_K W$ by \cite[Prop 1.1.31]{Eme17} if $V$ and $W$ are semi-complete.
We will write $V \hat{\otimes}_K W$ for completion for the case.

If $V$ is a Banach space, then we will write $V^\circ$ as its unit ball consisting of vectors with norms at most $1$.

If $G$ is a locally profinite group, and $U \subset G$ is an open compact subgroup, then we write $\cH(G, U)$ for the algebra of compactly supported, $U$-biinvariant functions $f : G \to \Z$, with multiplication given by convolution with respect to the Haar measure on $G$ which gives $U$ volume 1. 
If $X \subset G$ is a compact $U$-biinvariant subset, then we write $[X]$ for the characteristic function of $X$, an element of  $\cH(G, U)$.

If $L$ is a $p$-adic local field with a finite extension $L \subset K$, $V$ is a Hausdorff convex $K$-vector space and $X$ is a locally-$L$ analytic manifold, then we write $C^\la(X, V)$ for the $V$-valued locally analytic functions on $X$ (\cite[\S 2.1]{Eme17}).
We write $C^\la_c(X, V)$ for the subspace of functions with compact supports.
For a locally $L$-analytic group $G$, we write $\cD(G, K) := C^\la(G, K)'_b$ as the $K$-locally analytic distribution algebra of $G$ (\cite[\S 2]{ST02J}, denoted as $\cD^\la(G, K)$ in \cite[\S 5]{Eme17}).

If $H$ is a locally ($\QQ_p$-)analytic group and $K$ is a $p$-adic local field, 
we use $\Ban_H(K)$ to denote the additive category of $K$-Banach representation of $H$ (with $H$-equivariant continuous maps as morphisms),
and $\mathrm{Ban}^{\text{ad}}_H(K)$ to denote the abelian category of admissible continuous Banach $H$-representations (\cite[\S 2]{ST02I}, \cite[\S 6.2, Cor. 6.2.16]{Eme17}).

If $G$ is a connected linear algebraic group over a number field $F$, we refer to $X^G$ as the symmetric space for $\Res_{F/\QQ} G$ in the sense of \cite{BS73}[\S 2] and \cite{NT15}[\S 3.1].
And let $\overline{X}^G$ denote the partial Borel–Serre compactification of $X^G$ (\cite{BS73}[\S 7.1]).
We also define $\partial X^G := \overline{X}^G - X^G$ and \[ \mathfrak{X}_{\mathrm{G}}:= G(F) \backslash\left(X^G \times G(\AA_{F}^{\infty})\right), ~~ \partial \mathfrak{X}_{\mathrm{G}}:= G(F) \backslash\left(\partial X^G \times G(\AA_{F}^{\infty})\right). \]

\section{Local results}\label{LR}
Let $\cO$ be a topological ring, $X$ be a topological space, and $M$ be a topological $\cO$-module. 
We write $\cC(X, M)$ as the topological $\cO$-module of continuous maps from $X$ to $M$ with compact open topology.
By results in \cite[Tag 0B1Y]{Sta}, projective limits exist in the category of topological $\cO$-modules.
\begin{lem}\label{top projlim}
For $M \simeq \varprojlim\limits_{i} M_i$ as an isomorphism in the category of topological $\cO$-modules, there is a canonical isomorphism $\cC(X, M) \simeq \varprojlim\limits_{i} \cC(X, M_i)$ in the category of topological $\cO$-modules.
\end{lem} 

\begin{proof}
By universal property, there is a canonical map $\cC(X, M) \to \varprojlim\limits_{i} \cC(X, M_i)$, which is injective.
By \cite[Lem. 5.30.11, Tag 0B1Y]{Sta}, $M \simeq \varprojlim\limits_{i} M_i$ also holds as abstract $\cO$-modules.
Given a sequence of compatible continuous functions $(f_i)_i, ~ f_i \in \cC(X, M_i)$, we hence obtain a function $f: X \to M$ by $M \simeq \varprojlim\limits_{i} M_i$.
Since $M$ has the limit topology, $f \in \cC(X,M)$ is continuous.
This constructs the reverse arrow $\varprojlim\limits_{i} \cC(X, M_i) \to \cC(X, M)$. 
Let $p_i: M \to M_i$ be the natural projection, and let $U$ be an open in $M_i$.
One can check that it is continuous by checking it on a topological basis $\cV(K, p_i^{-1}(U))$ (space of continuous functions carrying $K$ to $p_i^{-1}(U)$) for any compact $K \subset X$, we're done.
\end{proof}

Let $G$ be a topological group with a closed subgroup $P$.
Let $V$ be a topological $\cO$-module with continuous $P$ action ($P \times V \to V$ is continuous).

We form the continuous induction
$$\operatorname{ct-Ind}^G_P(V):=\{f: G \to V | f ~ \mathrm{continuous}, ~ f(gp)=p^{-1}\cdot f(g), ~ \forall p \in P, g\in G\},$$ endowed with subspace topology of $\cC(G, V)$.

Let $L$ be a finite extension of $\QQ_p$, and let $K$ be an extension of $L$.
For the rest of the section, let $G$ be a locally $L$-analytic group with an analytic subgroup $P$. 
Suppose $G_0$, $P_0$ are compact subgroups of $G$, $P$ such that $G_0/ P_0 \xrightarrow{\sim} G/P$.

\begin{prop}
If $V$ is a $K$-Banach representation of $P$, the group action $G \times \operatorname{ct-Ind}^G_P(V) \to \operatorname{ct-Ind}^G_P(V)$ is continuous.
\end{prop}

\begin{proof}
It is reduced to the fact that the group action $G \times \cC(G, V) \to \cC(G, V)$ is continuous by \cite[Lem 3.1.1, Prop 3.1.5]{Eme17}.
\end{proof}

\begin{lem}\label{Ind ad}
If $V$ is a $K$-Banach representation of $P$, then $\operatorname{ct-Ind}^G_P(V)$ is a $K$-Banach representation of $G$.
Moreover, if both $P$ and $G$ are compact and $V$ is admissible (\cite{ST02I}, \cite{Eme17}), then $\operatorname{ct-Ind}^G_P(V)$ is admissible as a $K$-Banach representation of $G$.
\end{lem}
\begin{proof}
There is a topological isomorphism $\operatorname{ct-Ind}^{G_0}_{P_0}(V) \xrightarrow{\sim} \operatorname{ct-Ind}^G_P(V)$.
When $V$ is Banach, the topology of $\cC(G_0, V)$ coincides with the maximal norm, where $\cC(G_0, V)$ is Banach.
By \cite[Prop/Def 6.2.3]{Eme17}, if $V$ is admissible then there exists a $P_0$-equivariant closed embedding $V \to \cC(P_0, K)^n$ for some $n$.
The admissibility is therefore guaranteed by the $G_0$-equivariant closed embedding $\operatorname{ct-Ind}^{G_0}_{P_0}(V) \to \operatorname{ct-Ind}^{G_0}_{P_0} (\cC(P_0, K)^n) \simeq \cC(G_0, K)^n$. 
\end{proof}

\begin{prop}\label{ind projlim}
The continuous induction functor commutes with projective limits. Equivalently, 
$$\operatorname{ct-Ind}^G_P(\varprojlim\limits_{i} V_i) \simeq \varprojlim\limits_{i} \operatorname{ct-Ind}^G_P(V_i)$$ for a projective system $(V_i)$ of topological $\cO$-modules with continuous $P$ action.
\end{prop}

\begin{proof}
By \cite[4.1.1]{Fea99}, the map $G \to G/P$ admits an analytic splitting $\imath: G/P \times P \to G$ as locally analytic manifolds.
If $V$ is a topological $\cO$-module with continuous $P$ action, $\operatorname{ct-Ind}^G_P(V) \simeq \cC(G/P, V)$ as topological $\cO$-modules.
The claim then follows from Lemma \ref{top projlim}.
\end{proof}

\begin{lem}
Any locally analytic representation of $G$ on a vector space of compact type over $K$ restricts to a locally analytic representation of $P$.
\end{lem}

\begin{proof}
If $V$ is such a representation of $G$, then the orbit map gives us $V \hookrightarrow C^{\la}(G, K) \hat{\otimes} V$.
Locally analytic functions on $G$ pull back to locally analytic functions on $P$, there is a induced orbit map $V \hookrightarrow C^{\la}(P, K) \hat{\otimes} V$ for $P$.
\end{proof}

For a locally analytic representation $V$ of $P$, we have the locally analytic induction as follows (\cite[\S 4.1.1]{Fea99}):
$$\operatorname{la-Ind}^G_P(V):=\{f: G \to V | f ~ \text{locally analytic}, ~ f(gp)=p^{-1}\cdot f(g), ~ \forall p \in P, g\in G\}.$$
\begin{thm}\label{Ind la}
Let $V$ be a $K$-Banach space equipped with a continuous $P$-action.
The locally analytic vectors of $\operatorname{ct-Ind}^G_P(V)$ are identified with the locally analytic induction of locally analytic vectors of $V$ as a topological isomorphism:
$$(\operatorname{ct-Ind}^G_P(V))^{\la} \simeq \operatorname{la-Ind}^G_P(V^{\la}).$$
\end{thm}

\begin{proof}
We adapt some notation in \cite{Eme17}, and use $\cI$ to denote $\operatorname{ct-Ind}^G_P(V)$.
By \cite[\S 3.5]{Eme17}, there is a continuous $P$-equivariant injection of locally convex Hausdorff $K$-spaces $V^{\la} \hookrightarrow V$.
We know $\cI$ is a $K$-Banach closed subrepresentation of the space of continuous $V$-valued functions $C(G, V)$ on $G$.
By \cite[\S 4.1.1]{Fea99}, $\operatorname{la-Ind}^G_P(V^{\la})$ is a closed subspace of $C^{la}(G, V^{\la})$.
There is a natural continuous injection $\operatorname{la-Ind}^G_P(V^{\la}) \hookrightarrow \cI$, following from continuous injections by \cite[Prop. 2.1.7, Prop. 2.1.26]{Eme17}, $$C^{\la}(G,V^{\la}) \hookrightarrow C(G,V^{\la}) \hookrightarrow C(G,V).$$
Let $H$ be an analytic open subgroup of $G$, and let $\HH$, $\HH_P$ be rigid analytic group corresponding to $H$ and $H \cap P$.
Any group element $g \in G$ defines a map $E_g : \cI \to V$ sending $f \in \cI$ to $f(g) \in V$.
Consider $\cI_{\HH-\an} \to C^{\an}(\HH, \cI) \xrightarrow{E_g} C^{\an}(\HH, V)$, for any $\HH$-analytic vector $f \in \cI_{\HH-\an}: G \to V$, as a map it is $\HH$-analytic by evaluating at different points.
Therefore the map $H \times (P \cap H) \xrightarrow{m} H \xrightarrow{f} V$ is $\HH \times \HH_P$-analytic. 
Such a map is represented as an element in 
\begin{equation}\label{aniso}
\begin{split}
C^{\an}(\HH \times \HH_P, V) & \simeq C^{\an}(\HH \times \HH_P, K) \hat{\otimes} V \\
& \simeq C^{\an}(\HH, K) \hat{\otimes} C^{\an}(\HH_P, K) \hat{\otimes} V\\
& \simeq C^{\an}(\HH, C^{\an}(\HH_P, V)), 
\end{split}
\end{equation}
we know that $V_{\HH_P - \an}$ is a closed subspace of $C^{\an}(\HH_P, V)$.
Both $V$ (with orbit map) and $C^{\an}(\HH_P, V)$ are embedded into $C(\HH_P, V)$.
For $f \in \cI_{\HH-\an}: G \to V$, we have $f(gp)=p^{-1}f(g)$ for $p \in P$ from definition. 
Evaluating at identity of $P \cap H$, the target of $f$ lies in $V \cap C^{\an}(\HH_P, V) = V_{\HH_P - \an}$ by \cite[Thm 3.3.16]{Eme17}.
$f$ induces an analytic map $H \to V_{\HH_p - \an}$ by (\ref{aniso}). 
We can claim the topological isomorphism $V^{\la}_{\HH_P - \an} \simeq V_{\HH_P - \an}$ form \cite[Prop 3.4.14]{Eme17}.
Then by the same argument we know $(\operatorname{la-Ind}^G_P(V^{\la}))_{\HH - \an} \xrightarrow{\sim} \cI_{\HH - \an}$, hence $(\operatorname{la-Ind}^G_P(V^{\la}))^{\la} \xrightarrow{\sim} \cI^{\la}$.
We finally win by applying \cite[Thm 3.6.12]{Eme17}.
\end{proof}

Recall that from notation $\mathrm{Ban}_G(K)$ is the category of $K$-Banach representation of $G$, and $\mathrm{Ban}^\ad_G(K)$ is the abelian category of admissible continues $K$-Banach representations of $G$ considered in \cite[\S 6.2]{Eme17}.
We use $\mathrm{Rep}^\mathrm{a}_K(G)$ to denote the abelian category of admissible locally analytic representations over $K$ in \cite[\S 6.1]{Eme17} (and equivalently in \cite[\S 6]{ST03} in the compact case).
We use $\replag$ to denote the category of locally analytic $G$ representations on compact type locally convex topological $K$-vector spaces.

\begin{lem}\label{en inj}
If $G$ is compact, $\mathrm{Ban}^\ad_G(K)$ has enough injective objects.
\end{lem}

\begin{proof}
If $G$ is compact, $\mathrm{Ban}^\ad_G(K)$ is anti-equivalent to the category of finitely generated (left unital) $K[[G]]$-modules via duality by \cite[Thm 3.5]{ST02I}, which has enough projective objects.
\end{proof}

As it is remarked in the (end of) proof of \cite[Prop. 2.1.2]{Eme07}, continuous parabolic induction preserves admissibility for $p$-adic reductive groups.

Now fix a connected quasi-split reductive linear algebraic group $\GG$ over a $p$-adic field $L$, as well as a Borel pair $(\BB, \NN)$ with Levi decomposition $\BB = \TT \NN$. Let $\overline{\NN}$ be the opposite unipotent of $\NN$.
Write $G = \GG(L)$, $B = \BB(L)$, $T = \TT(L)$, and $N = \NN(L)$, $\overline{N} = \overline{\NN}(L)$.
Let $G_0$ be a compact open subgroup of $G$ and $N_0 := G_0 \cap N$, $\overline{N}_0 := G_0 \cap \overline{N}$, $T_0 := G_0 \cap T$. 
Let $T^+$ be the submonoid of $T$ which stabilizes the closed subgroup $N_0 \subset G$.

Let $\reptopt$ be the additive category whose objects are Hausdorff locally convex $K$-vector spaces of compact type, equipped with a topological action of $T^+$, and whose morphisms are continuous $T^+$-equivariant $K$-linear maps (\cite[\S 3.1]{Eme06A}, by a “topological action” of $T^+$ on a
topological vector space $V$ we mean an action of $T^+$ on $V$ by continuous endomorphisms).
Consider the following left exact functor $\Gamma^{\la}_{N_0}: \mathrm{Ban}^\ad_G(K) \to \mathrm{Rep}^\mathrm{top.c}_K(T^+)$ defined as $\Gamma_{N_0} \circ \mathrm{la}$, where 
\begin{eqnarray*}
\mathrm{la}: \mathrm{Ban}^\ad_G(K) & \to & \mathrm{Rep}^\mathrm{a}_K(G) \\
V & \mapsto & V^{\la}
\end{eqnarray*}
is the exact functor of passing to locally analytic vectors (\cite[Thm 7.1]{ST03}). 
We define Hecke action $T^+$ on $V^{N_0}$ to be 
\[ t\cdot v := [N_0:tN_0t^{-1}]^{-1} \sum_{n \in [N_0:tN_0t^{-1}]} ntv, ~\for ~ \forall t \in T^+, \forall v \in V^{N_0}.\]
We refer to \cite[\S 3.4]{Eme06A} for the Hecke action of $T^+$.
\begin{eqnarray*}
\Gamma_{N_0}: \mathrm{Rep}^\mathrm{a}_K(G) & \to & \reptopt \\
V & \mapsto & V^{N_0}
\end{eqnarray*} is the left exact functor of taking $N_0$-invariants.

If $V \in \mathrm{Ban}^\ad_{G_0}(K)$, then $V^\la$ is a $D(G_0, K)$ module by \cite[Prop 3.2]{ST02J}.
\begin{lem}\label{la acyc}
Let $\frakn$ be the Lie algebra of $N_0$ over $L$.
If $I$ is an injective object in $\mathrm{Ban}^\ad_{G_0}(K)$, then $H^i(\frakn, I^\la)$ vanish for $i \geq 1$.
\end{lem}
\begin{proof}
Let $I^\ast$ be the corresponding finitely generated $K[[G_0]]$-module, $I^\ast$ is projective.
There exists a complement $M^\ast$ such that $I^\ast \oplus M^\ast \simeq K[[G_0]]^r$ for some rank $r \in \ZZ_{\geq 1}$.
As the dual of $K[[G_0]]$ is space of continuous functions $C(G_0, K)$ on $G$, there exists a closed complement $M$ such that $I \oplus M \simeq C(G_0,K)^r$.
By taking locally analytic vectors, $I^\la \oplus M^\la \simeq C^\la(G_0, K)^r$.

As locally analytic manifolds, $G_0$ decomposes as $N_0 \times N_0 \backslash G_0$, $$C^\la(G_0, K) \simeq C^\la(N_0, K) \hat{\otimes}_K C^\la(N_0 \backslash G_0, K).$$
Since $\bullet \otimes_K C^\la(N_0 \backslash G_0, K)$ preserves exactness of locally convex $K$-vector spaces, so does $\bullet \hat{\otimes}_K C^\la(N_0 \backslash G_0, K)$ by Prop \ref{HS comp exact}.
It suffices to prove $H^i(\frakn, C^\la(G_0, K))$ vanish for $i \geq 1$, which turns out to be the Poincar\'e lemma (Proof of \cite[Prop 3.1]{ST05} as a $p$-adic analogue of \cite[VII.1.1]{BW80}).
\end{proof}

\begin{cor}
Let $\for$ be the forgetful functor \[ \for: \reptopt \to \mathrm{Mod}(K[T^+]). \] For a complex $C_\bullet \in \bK^b(\mathrm{Ban}^\ad_{G_0}(K))$, we choose an injective resolution $C_\bullet \to \cI_\bullet$ (the existence follows from Lem \ref{en inj}). 
We have $R^i (\for \circ \Gamma^\la_{N_0}) (\cI_\bullet)$ (derived functors for $\for \circ \Gamma^\la_{N_0}$) coincide with $\Ext^i_{D(N_0, K)}(\mathds{1}, \cI^\la_\bullet)$ (derived functors of Hom functor) as abstract $K[T^+]$-modules.
\end{cor}
\begin{proof}
By \cite[Thm 4.8, Thm 4.10]{Koh11}, $$\Ext^i_{D(N_0, K)}(\mathds{1}, V^\la) \simeq H^i(\frakn, V^\la)^{N_0}$$ for any $V \in \mathrm{Ban}^\ad_{G_0}(K)$. 
Therefore  $\Ext^i_{D(N_0, K)}(\mathds{1}, I^\la)$ vanish for any injective object $I$ in $\mathrm{Ban}^\ad_{G_0}(K)$ and $i \geq 1$ by Lem \ref{la acyc}.
Both delta-functors are effaceable and calculate cohomology of $(\cI_\bullet^\la)^{N_0}$.
\end{proof}

\begin{df}\label{Gaction}
Suppose $C_\bullet \in \bK^b(\contadg)$. 
By a $G$-extension of $C$ we mean the data $(C', i, p)$ such that $C'_\bullet \in \bK^b(\mathrm{Ban}_G(K))$, with a $G_0$-equivariant homotopy equivalence $C_\bullet \stackrel[i]{p}{\leftrightarrows} C'_\bullet$ in $\bK^b(\mathrm{Ban}_{G_0}(K))$. 
\end{df}

In particular, $H^\ast(p \circ i) = \mathrm{id}$, $H^\ast(i \circ p) = \mathrm{id}$.

If $C_{1,\bullet}, C_{2,\bullet}$ are two homotopic equivalent objects in $\bC\bh^b(\contadg)$ and $C'_{1,\bullet}, C'_{2,\bullet}$ are two homotopic equivalent objects in $\bC\bh^b(\mathrm{Ban}_G(K))$, then a given $G$-extension $C'_{2,\bullet}$ of $C_{2,\bullet}$ transfer to a $G$-extension $C'_{1,\bullet}$ of $C_{1,\bullet}$.

\begin{df}\label{ext of mor}
For a morphism $f : C_{1,\bullet} \to C_{2,\bullet}$ in $\bK^b(\contadg)$, by a $G$-extension $f'$ of $f$ we mean $G$-extensions $(C_{1,\bullet}, i, p)$ of $C_{1,\bullet}$, $(C_{2,\bullet}, j, q)$ of $C_{2,\bullet}$ and $f' : C'_{1,\bullet} \to C'_{2,\bullet}$ in $\bK^b(\mathrm{Ban}_G(K))$, with the following diagrams in $\bK^b(\mathrm{Ban}_{G_0}(K))$.
$$\xymatrix{
C_{1,\bullet} \ar[r]_i \ar[d]^f & C'_{1,\bullet} \ar[d]^{f'} \\
C_{2,\bullet} \ar[r]_j & C'_{2,\bullet} \\}
\xymatrix{
C_{1,\bullet} \ar[r]_i \ar[d]^f & C'_{1,\bullet} \ar[d]^{f'} \\
C_{2,\bullet} & C'_{2,\bullet} \ar[l]_q\\}
\xymatrix{
C_{1,\bullet} \ar[d]^f & C'_{1,\bullet} \ar[d]^{f'} \ar[l]_p\\
C_{2,\bullet} \ar[r]_j & C'_{2,\bullet} \\}
\xymatrix{
C_{1,\bullet} \ar[d]^f & C'_{1,\bullet} \ar[d]^{f'} \ar[l]_p\\
C_{2,\bullet} & C'_{2,\bullet} \ar[l]_q\\}$$ 
for $G$-extension $C'_{1,\bullet}$ of $C_{1,\bullet}$ and $C'_{2,\bullet}$ of $C_{2,\bullet}$, commutative up to homotopy equivalences (these diagrams become commutative in the homotopy category).
\end{df}

Commutativity for these four diagrams are equivalent. 
For example, if $j \circ f$ is homotopic equivalent to $f' \circ i$, then $f$ is homotopic equivalent to $q \circ f' \circ i$.

Any $G$-extension $(C', i, p)$ of $C_\bullet \in \bK^b(\contadg)$ defines a map $G \to \mathrm{Aut}_{\bK^b(\mathrm{Ban}_K)} C_\bullet$ compatible with the natural action of $G_0$ (a possible misunderstanding is that a $G$-extension induces $G \to \mathrm{Aut}_{\bK^b(\contadg)} C_\bullet$, this is not true as a general element of $G$ does not commute with $G_0$).
If we fix a choice of representative $(C'_\bullet, i, p)$, then 
for any group element $g \in G$, we define a representative $\tilde{g} : C_\bullet \to C_\bullet$ to be $p \circ g \circ i$ lifting $g \in \mathrm{Aut}_{\bK^b(\mathrm{Ban}_K)} C_\bullet$.
If $g \in G_0$, this lift is given by $G_0$ structure regarding $C_\bullet$ as a complex of $G_0$ Banach representations.

It is easy to establish some first properties of such an extension structure.
\begin{prop}\label{first props of Gext}
We fix representatives of $G$-extension $(C'_\bullet, i, p)$ of $C_\bullet \in \bC\bh^b(\contadg)$.
\begin{itemize}
\item For any two $g_1, g_2 \in G$, $\tilde{g}_1 \circ \tilde{g}_2$ is homotopic equivalent to $\widetilde{g_1 g_2}$.
\item Different choices of $p : C'_\bullet \to C_\bullet$ with $i$ fixed or different choices of $i : C_\bullet \to C'_\bullet$ with $p$ fixed induce different $\tilde{g}$, but they all agree up to homotopy equivalences.
\item We have $\tilde{g}_1 \circ \tilde{g}_2 = \widetilde{g_1 g_2}$ if either $g_1$ or $g_2$ belongs to $G_0$.
\end{itemize}
In particular, a $G$-extension $(C'_\bullet, i, p)$ of $C_\bullet$ induces admissible Banach representations of $G$ on cohomology.
In other words, $H^i(C_\bullet) \in \Ban^{\ad}_{G}(K)$ for any $i$, extending the natural $G_0$ action to a $G$ action.
\end{prop}
\begin{proof}
The first two claims are clear by \cite[Tag 010W]{Sta}. The last one holds since both $i$ and $p$ are $G_0$-equivariant.
\end{proof}

\begin{prop}\label{Ind ext}
If $\widetilde{\GG}$ is a quasi-split reductive linear algebraic group over $L$, with a parabolic subgroup $\PP$, such that $\GG$ is a Levi of $\PP$.
Let $\widetilde{G} = \widetilde{\GG}(L)$ with compact open subgroup $\widetilde{G}_0$.
And let $P = \PP(L)$, $P_0 = \widetilde{G}_0 \cap P$, $G = \GG(L)$, $G_0 = \widetilde{G}_0 \cap G$.
Let $C_\bullet \in \bK^b(\contadg)$ with a $G$-extension $(C'_\bullet, i, p)$. By inflation we view $C_\bullet \in \bK^b(\Ban^{\ad}_{P_0}(K))$ with a $P$-extension $(C'_\bullet, i, p)$.
Suppose that the projection of $\PP$ onto $\GG$ sends $P_0$ inside $G_0$ and $\wtG_0 P = \wtG$, then there is a natural $\widetilde{G}$-extension ($\Ind^{\widetilde{G}}_{P} C'_\bullet, \tilde{i}, \tilde{p}) \in \bK^b(\Ban_{\widetilde{G}}(K))$ of $\Ind^{\widetilde{G}_0}_{P_0} C_\bullet \in \bK^b(\Ban^\ad_{\widetilde{G}_0}(K))$.
\end{prop}
\begin{proof}
By Lem \ref{Ind ad}, $\Ind^{\widetilde{G}}_{P} C'_\bullet \in \bK^b(\Ban_{\widetilde{G}}(K))$ and $\Ind^{\widetilde{G}_0}_{P_0} C_\bullet \in \bK^b(\Ban^\ad_{\widetilde{G}_0}(K))$.
We define 
$$\Ind^{\widetilde{G}_0}_{P_0} C_\bullet  \stackrel[\tilde{i}]{\tilde{p}}{\leftrightarrows}  \Ind^{\widetilde{G}}_{P} C'_\bullet ,$$ where $\tilde{i}$ sends $f$ to an unique function $f'$ in $\Ind^{\widetilde{G}}_{P} C'_\bullet$ such that $f'(g) = i(f(g))$ for $i: C_\bullet \to C'_\bullet$ and any $g \in \widetilde{G}_0$, $\tilde{p}$ sends $f'$ to $f$ satisfying $f(g) = p(f'(g))$ for $p: C'_\bullet \to C_\bullet$ and any $g \in \widetilde{G}_0$.
We check that $\tilde{i}$ is well-defined as if $g_1 p_1 = g_2 p_2$ for $g_1, g_2 \in \widetilde{G}_0$ and $p_1, p_2 \in P$, then $g_2^{-1} g_1 = p_2 p_1^{-1} \in \widetilde{G}_0 \cap P = P_0$,
$$f'(g_1 p_1) = p_1^{-1} i(f(g_1)) = p_2^{-1} g_2^{-1} g_1 i(f(g_1)) = p_2^{-1} i(f(g_2)) = f'(g_2 p_2).$$ 
It can also be checked that $\tilde{i}, \tilde{p}$ are $\widetilde{G}_0$-equivariant, and $\tilde{i} \circ \tilde{p}, \tilde{p} \circ \tilde{i}$ are both homotopic to identity. 
We check for $\tilde{i} \circ \tilde{p}$ as follows:
note that for each $n \in \ZZ$ we have a $G_0$-equivariant map $h_n: C'_n \to C'_{n-1}$ such that $$i \circ p - 1 = d_{n-1} \circ h_n + h_{n+1} \circ d_n,$$ where $d_n : C'_n \to C'_{n+1}$ is the differential map for $C'_\bullet$ at degree $n$.
Let $\tilde{d}_n : \Ind^{\widetilde{G}_0}_{P_0} C'_n \to \Ind^{\widetilde{G}_0}_{P_0} C'_{n+1}$ and $\tilde{h}_n : \Ind^{\widetilde{G}_0}_{P_0} C'_n \to \Ind^{\widetilde{G}_0}_{P_0} C'_{n-1}$ be the natural maps induced from $d_n$ and $h_n$.
We therefore define $\tilde{d}_n$ and $\tilde{h}_n$ on $\Ind^{\widetilde{G}}_{P} C'_\bullet$ via the natural $\widetilde{G}_0$-equivariant isomorphism $\Ind^{\widetilde{G}}_{P} C'_\bullet \xrightarrow{\sim} \Ind^{\widetilde{G}_0}_{P_0} C'_\bullet$, with the identity 
$$\tilde{i} \circ \tilde{p} - 1 = \tilde{d}_{n-1} \circ \tilde{h}_n + \tilde{h}_{n+1} \circ \tilde{d}_n.$$
\end{proof}
We do not assert functoriality for induction of $G$-extension of the above proposition.

\begin{lem}\label{Ind Inj}
Let notation be as in Prop \ref{Ind ext}.
If $I$ is an injective object in the category of admissible continuous Banach representations of $P_0$, then $\Ind^{\wtG_0}_{P_0} I$ is an injective object in the category of admissible continuous Banach representations of $\wtG_0$.
\end{lem}
\begin{proof}
As in the proof of Lem \ref{la acyc}, there exists a complement $M$ such that $I \oplus M \simeq C(P_0,K)^r$.
It suffices to prove $\operatorname{ct-Ind}^{\wtG_0}_{P_0} C(P_0,K) \simeq C(\wtG_0, K)$ as $\wtG_0$-representations.
There is a natural $\wtG_0$-equivariant map $\operatorname{ct-Ind}^{\wtG_0}_{P_0} C(P_0,K) \to C(\wtG_0, K)$ by sending $\tilde{f} \mapsto f$ such that $f(g) = \tilde{f}(g)(\mathrm{identity})$ for any $g \in \wtG_0$.
Any $f$ defines a $P_0$-equivariant function on $\wtG_0$ valued in $C(P_0, K)$ and vice versa, we get the desired isomorphism. 
\end{proof}

Suppose we have a $G$-extension (with fixed representatives) $(C'_\bullet, i, p)$ of $C_\bullet \in \bC\bh^b(\contadg)$.
For any $t \in T^+$, we define (derived) Hecke operators as follows. 

\begin{lem}\label{stab la}
For any $g \in G$, $\tilde{g}$ stabilizes $C^\la_\bullet$ and acts continuously on it.
\end{lem}
\begin{proof}
Taking locally analytic vectors is functorial, each of $i$, $g$, $p$ preserves locally analytic vectors (\cite[\S 3]{Eme17}).
\end{proof}
\begin{lem}\label{Ut action}
The operator $$U_t := \frac{1}{[N_0:tN_0 t^{-1}]} \sum_{n \in N_0 / tN_0 t^{-1}} \widetilde{nt} \in \End_K ((C_\bullet^{\la})^{N_0}),$$ where $\widetilde{nt} = p \circ nt \circ i$, is a well-defined endomorphism of $(C_\bullet^{\la})^{N_0}$.
\end{lem}
\begin{proof}
For any $v \in C_i^{N_0}$ and $n' \in N_0$, since $\tilde{n} \circ \tilde{g} = \tilde{ng}$, $\tilde{g} \circ \tilde{n} = \tilde{gn}$ for any $n \in N_0$, $g \in G$, we have
$$n' \cdot \sum_{n \in N_0 / tN_0 t^{-1}} \tilde{nt} (v) = \sum_{n \in N_0 / tN_0 t^{-1}} \tilde{nt} \circ n_k (v) = \sum_{n \in N_0 / tN_0 t^{-1}} \tilde{nt} (v),$$ where the sets $\{n' n t\} = \{ntn_k\}$ are equal (a permutation of elements), and all $n_k \in N_0$.
By the previous lemma, the operator stabilizes the closed subcomplex $(C_\bullet^{\la})^{N_0}$ of $C_\bullet^{\la}$.
\end{proof}
Let $\widehat{T}$ be the rigid analytic space over $L$ of locally $L$-analytic characters on $T$ (\cite[Prop 6.4.5]{Eme17}), and let $C^{\an}(\widehat{T}, K)$ denote the nuclear Fr\'echet algebra of $K$-valued
rigid analytic functions on $\widehat{T}$.
Evaluation induces a natural map $T \to C^{\an}(\widehat{T}, K)$.
Pick $t \in T$ such that $\cap_{n \geq 1} t^n N_0 t^{-n} = \{1\}$, we may replace $t$ by a sufficient large power of $t$ to make it satisfy the assumption in \cite[Prop 4.1.6 (vi)]{Eme06A}.
Let $Y^+$ (resp. $Y$) be the submonoid (subgroup) generated by $T_0$ and $t$.
\begin{df}\label{FSP}
If $V$ is any object of $\mathrm{Rep}^{\mathrm{top.c}}_K(Y^+)$, then we write
\[ V_{\fs} := \cL_{b, Y^+} (C^{\an}(\widehat{Y}, K), V) \in \mathrm{Rep}^{\la}_K(Y). \]
The subscript refers to ``finite slope part" of \cite[Def. 3.2.1]{Eme06A}. 
\end{df}

Given a $G$-extension $(C'_\bullet, i, p)$ of $C_\bullet$, and $t \in T$, there is a $U_t$ action on $(C_\bullet^{\la})^{N_0}$ by Lem \ref{Ut action}.
We prescribe $t^i$ ($i \geq 0$) action as $(U_t)^i$. 
By \cite[Lem 3.4.4]{Eme06A}, $U_{t^i} = (U_t)^i$ for $i \geq 0$ up to homotopy equivalence.
By (3) of Prop \ref{first props of Gext}, $U_t$ operator commutes with $T_0$, therefore an authentic action of $Y^+$ on $(C_\bullet^{\la})^{N_0}$ is defined:
for any $y = t_0 t^i \in Y^+$, $t\cdot v := t_0 (U_t)^i \cdot v$ for any $v \in (C_\bullet^{\la})^{N_0}$.

We define Jacquet functor $J_B$ on the $G$-extension $(C'_\bullet, i, p)$ of $C_\bullet$ to be $$J_B(C_\bullet) := \left((C_\bullet^{\la})^{N_0}\right)_\fs \in \bC\bh(\mathrm{Rep}^{\la}_K(Y)),$$
where $\fs$ is with respect to the semigroup generated by $U_t$. 
By Lem \ref{Ut action}, for any $t \in T^+$, $t$ acts on $(C^{\la}_\bullet)^{N_0}$ via the Hecke operator $U_t$. 
These operators do not commute and are not associative in general, but they commute up to homotopy equivalences.
The Hecke operators associated to the monoid $T^+$ defines a topological action on cohomology $H^\ast(J_B(C_\bullet))$ by \cite[Lem 3.4.4]{Eme06A}.
The fact that cohomology of $J_B(C_\bullet)$ are independent of various choices follows from Prop \ref{GAEA} and Cor \ref{fs indep} proven in the next section.

If $Z$ is a topologically finitely generated abelian locally $L$-analytic group,
let $\Rep_{\mathrm{es}}(Z)$ be the abelian category of essentially admissible locally analytic representations of $Z$ over $K$ (\cite[\S 6.4]{Eme17}).

\begin{lem}\label{la ext}
If $W$ is a locally analytic $T$-representation, then the action of $T$ on $W$ extends to an action of $K[T] \to C^{\an}(\widehat{T}, K)$ such that $C^{\an}(\widehat{T}, K) \times W \to W$ is separately continuous.
\end{lem}
\begin{proof}
Let $\TT_n$ be a sequence of abelian rigid analytic group corresponding to a cofinal sequence of analytic open subgroups $\TT_n(L)$ of $T_0$.
Let $W_n$ be the BH-subspace of $W$ obtained as the image of the natural map $W_{\TT_n-\an} \to W$.
Clearly $W = \bigcup_{n=1}^\infty W_n$, and each $W_n$ is invariant by $T$-action since $T$ is abelian.
$W$ satisfies condition (iii) of \cite[Prop 6.4.7]{Eme17}, and $W$ admits an action of $C^{\an}(\widehat{T}, K)$.
\end{proof}

\begin{lem}\label{sub ea}
If $W$ is a locally analytic $T$-representation over $K$ such that it restricts to an essentially admissible representation of $Y$, then $W$ is an essentially admissible representation of $T$.
\end{lem}
\begin{proof}
By definition, $W'_b$ is coadmissible as a $C^{\an}(\widehat{Y}, K) \hat{\otimes}_K \cD^{\la}(T_0, K)$-module.
By \cite[Prop. 6.4.6]{Eme17}, $\cD^{\la}(T_0, K)$ is identified with $C^{\an}(\widehat{T_0}, K)$ as a subalgebra of $C^{\an}(\widehat{Y}, K)$ for $T_0 \hookrightarrow Y$.
We have that $W'_b$ is coadmissible as a $C^{\an}(\widehat{Y}, K)$-module (\cite[Prop 2.3.2]{Eme06I}, remarks before \cite[Prop. 6.4.10]{Eme17}).

Let $C^{\an}(\widehat{Y}, K) \xrightarrow{\sim} \varprojlim_n A_n$ be a Fr\'echet-Stein structure of $C^{\an}(\widehat{Y}, K)$ rising from an admissible cover $\{\widehat{Y_n}\}$ of $\widehat{Y}$ with $A_n := \cO(Y_n)$.
Then each $A_n \hat{\otimes}_{C^{\an}(\widehat{Y}, K)} W'_b$ is finitely generated as an $A_n$-module.

By \cite[Prop 6.4.7]{Eme17} and the previous Lem \ref{la ext}, the $C^{\an}(\widehat{Y}, K)$ action on $W'_b$ extends (separately) continuously to $C^{\an}(\widehat{T}, K)$.
Therefore $A_n \hat{\otimes}_{C^{\an}(\widehat{Y}, K)} W'_b$ is finitely generated as a separately continuous $A_n \hat{\otimes}_{C^{\an}(\widehat{Y}, K)} C^{\an}(\widehat{T}, K)$-module.
We have \[ A_n \hat{\otimes}_{C^{\an}(\widehat{Y}, K)} C^{\an}(\widehat{T}, K) \to \End_{A_n} (A_n \hat{\otimes}_{C^{\an}(\widehat{Y}, K)} W'_b), \] factoring through affinoid algebras $B_n$ which are finite over $A_n$, defining coherent sheafs on $\widehat{Y_n} \times_{\widehat{Y}} \widehat{T}$ supported on affinoid spaces $\Sp(B_n) \hookrightarrow \widehat{Y_n} \times_{\widehat{Y}} \widehat{T}$ (as $B_n$ is a quotient of a Fr\'echet-Stein algebra by a closed ideal, one can prove this by proving that any continuous map of $K\{\{z_1,z_1^{-1},\cdots,z_m,z_m^{-1}\}\} \to B_n$ factors through an affinoid algebra of a $p$-adic annulus of certain radii).
These sheafs are compatible with respect to $n$, hence we obtain a coherent sheaf on $\widehat{T}$ whose global section agrees with $W'_b$ by varying $n$.
The claim then follows from coadmissibility of $W'_b$ with respect to the Fr\'echet-Stein algebra $C^{\an}(\widehat{T}, K)$.
\end{proof}

\begin{prop}\label{ES sketch}
We use notation after Def \ref{FSP}.
For each $i$, $J_B(C_i) = \left((C^{\la}_i)^{N_0}\right)_{\fs}$ is an essentially admissible locally analytic representation of $Y$, equivalently, \[ J_B(C_\bullet) \in \bC\bh(\Rep_{\mathrm{es}}(Y)).\]  
\end{prop}
\begin{proof}
This follows from Emerton's strategy \cite{Eme06A}, we give a summary (and a simplification in our setting) as a proof.

We refer to \cite{Eme06A}, \cite{Eme17} for related notation.
Let $V$ be $C^{\la}_i$.
Let $\{H_n\}_{n \geq 0}$ be a decreasing cofinal sequence of good analytic
open subgroups of $G_0$ satisfying conditions in \cite[Prop 4.1.6]{Eme06A}, with rigid analytic groups $\HH_n$ underlying $H_n$.
Let $T_n := H_n \cap T$, and $\TT_n$ be the rigid analytic group underlying $T_n$.
Let $\HH^\circ_n$, $\TT^\circ_n$ be the strictly $\sigma$-affinoid analytic open subgroup of $\HH_n$, $\TT_n$ (\cite[\S 4.1]{Eme06A}).
Let $D(\HH^\circ_n, H_0)$ be the strong dual to the nuclear Fr\'echet space $C^{\la}(H_0, K)_{\HH^\circ_n-\an}$.

For each $n \geq 0$, we define 
\[ U_n := \left(D(\HH^\circ_n, H_0) \hat{\otimes}_{D^{\la}(H_0, K)} V'_b\right)_{N_0},\]
as the Hausdorff $N_0$-coinvariants of the completed tensor product  \[D(\HH^\circ_n, H_0) \hat{\otimes}_{D^{\la}(H_0, K)} V'_b, ~ \mathrm{with} ~\mathrm{kernel} ~ K_n ~\mathrm{fitting ~ into} \] \begin{eqnarray}\label{SS Un} 0 \to K_n \to D(\HH^\circ_n, H_0) \hat{\otimes}_{D^{\la}(H_0, K)} V'_b \to U_n \to 0. \end{eqnarray}

Since $V$ consists of locally analytic vectors in an admissible Banach representation $C_i$ of $G_0$, $(C_i)'_b$ is a finitely presented $K[[H_0]]$-module over the Iwasawa algebra $K[[H_0]]$. 
By \cite[Thm 7.1]{ST03}, $V'_b \simeq D^{\la}(H_0, K) \otimes_{K[[H_0]]} (C_i)'_b$ is a finitely presented $D^{\la}(H_0, K)$-module.

The inverse system $\varprojlim_{n} K_n$ has dense transition maps as $(N_0 - 1) V'_b$ is dense in $D(\HH^\circ_n, H_0) \hat{\otimes}_{D^{\la}(H_0, K)} V'_b$ (and $D^{\la}(H_0, K)$ is dense in $D(\HH^\circ_n, H_0)$), we can apply a topological version of Mittag-Leffler condition used by \cite[\S Thm(B)]{ST03}, which refers to \cite[Chapter III, \S 0.13.2]{EGA61} so that \[ 0 \to \varprojlim_n K_n \to V'_b \to \varprojlim_n U_n \to 0.\]
The strong dual to $V^{N_0}$ is naturally identified with the Hausdorff $N_0$-coinvariants $(V'_b)_{N_0}$ of the strong dual to $V$,
we have (corresponding to \cite[(4.2.6)]{Eme06A}) a $Y^+$-equivariant isomorphism \[ (V^{N_0})'_b \xrightarrow{\sim} \varprojlim_n U_n. \]

The proof of \cite[Cor 4.2.26]{Eme06A} (replace $C^\an(\hat{Z}_{G, n}, K)^{\dagger} \hat{\otimes}_{K} D(\mathbb{M}_{n}^{\circ}, M_{0})$ by $D(\TT^\circ_n, T_0)$ and $C^\an(\hat{Z}_{G, n}, K)^{\dagger} \hat{\otimes}_{K} D(\mathbb{H}_{n}^{\circ}, H_{0})$ by $D(\mathbb{H}_{n}^{\circ}, H_{0})$) shows the map \[ D(\TT^\circ_n, T_0) \otimes_{D(\TT^\circ_{n+1}, T_0)} U_{n+1} \to U_n \] is a $D(\TT^\circ_n, T_0)$-compact map in the sense of \cite[Def 2.3.3]{Eme06A}.
The proof of \cite[Prop 4.2.28]{Eme06A} shows that for $t$ (suitably chosen in the sense of \cite[Prop 4.1.6 (vi)]{Eme06A}), the Hecke action of $t$ (explicitly expressed as (\ref{Hecke formula})) given by $G$-extension $(C'_\bullet, i, p)$ of $C_\bullet$ on $U_n$ factors through
\begin{eqnarray}\label{diag diag} \xymatrix{
D(\TT^\circ_n, T_0) \otimes_{D(\TT^\circ_{n+1}, T_0)} U_{n+1} \ar[rr] \ar[dd]^t & & U_n \ar[dd]^t \ar[lldd]\\
& & \\
D(\TT^\circ_n, T_0) \otimes_{D(\TT^\circ_{n+1}, T_0)} U_{n+1} \ar[rr] & & U_n. \\
}\end{eqnarray}

Let $H(t)_n := (t^{-1} H_n t) \cap H_n$, $H(t)_{n,n+1} := (t^{-1} H_n t) \cap H_{n+1}$ with underlying strictly $\sigma$-affinoid analytic groups $\HH(t)_n^\circ$, $\HH(t)_{n,n+1}^\circ$ as in \cite[\S 4.1]{Eme06A}.
We define \[ U(t)_{n,n+1} := \left(D(\HH(t)^\circ_{n,n+1}, H_0) \hat{\otimes}_{D^{\la}(H_0, K)} V'_b\right)_{N_0}, \]
\[ U(t)_{n} := \left(D(\HH(t)^\circ_{n}, H_0) \hat{\otimes}_{D^{\la}(H_0, K)} V'_b\right)_{N_0}. \]

Note that we have to modify the map (4.2.9) $U_n \to U(t)_n$ of \cite{Eme06A}: namely the map (4.2.16) should be changed to \begin{eqnarray}\label{Hecke formula} \mu \otimes v^{\prime} \mapsto \delta(z) \sum_{x} s_{n, k}^{\prime}(\mu) \otimes i^\ast \circ t^{-1} n^{-1} \circ p^\ast v', \end{eqnarray} where $i^\ast$, $p^\ast$ are the dual maps of $i$, $p$ (and we replace $z$, $x$ in (4.2.16) of \cite{Eme06A} by $t$, $n$ in line with notions in Lem \ref{Ut action}) in our setting.
We list the properties for Emerton's arguments to work: the dual of Hecke operator $U_t$ of Lem \ref{Ut action} on $U_n$ is continuous and $i^\ast$, $p^\ast$ are $G_0$-equivariant.  
The same arguments in \cite[Lem 4.2.11, Lem 4.2.19]{Eme06A} define the map (4.2.9) so that the diagram (4.2.10) of \cite{Eme06A} is commutative.
For the commutativity of (\ref{diag diag}), we decompose the diagram by the following one (diagram (4.2.29) of \cite{Eme06A}):
\[ \xymatrix{
D(\TT^\circ_n, T_0) \otimes_{D(\TT^\circ_{n+1}, T_0)} U(t)_{n,n+1} \ar[rr] \ar[dd]^t & & U(t)_n \ar[dd]^t \ar[lldd]\\
& & \\
D(\TT^\circ_n, T_0) \otimes_{D(\TT^\circ_{n+1}, T_0)} U_{n+1} \ar[rr] & & U_n. \\
} \]
Commutativity of the upper left triangle of (\ref{diag diag}) follows from directly verifying the formula (\ref{Hecke formula}).

We finally can apply \cite[Prop 3.2.24]{Eme06A} to conclude coadmissibility of $(V^{N_0}_\fs)'_b$ over $C^{\an}(\widehat{Y}, K)$, or equivalently, $V^{N_0}_\fs$ is an essentially admissible $Y$-representation.
\end{proof}

\section{Exactness of finite slope part and applications on Jacquet functors}\label{Exact FS}

In this section we establish an exactness result of taking finite slope part on locally convex vector spaces. 
We apply this result to define derived Jacquet functors and obtain further results on functoriality in \S \ref{App uni}. It is worth noting that Hao Lee defines derived Jacquet functors using a different approach.

We use notation in \S \ref{LR}.
For a locally convex vector space $V$ over a $p$-adic local field $K$, let $\overline{\{0\}_V}$ be the closure of $0$ in $V$, let $V / \overline{\{0\}_V}$ be the Hausdorff quotient of $V$, and let $\hat{V}$ be the Hausdorff completion of $V$ (\cite[\S 7]{Sch01}).

\begin{lem}\label{HS quot SES}
Let $0 \to U \to V \to W \to 0$ be a short exact sequence of locally convex spaces such that $U$ has the subspace topology and $W$ has quotient topology with respect to $V$.
Then $0 \to U / \overline{\{0\}_U} \to V / \overline{\{0\}_V} \to W / \overline{\{0\}_W} \to 0$ is exact at the first and last term with subspace and quotient topology.
For the middle term, the image of $U / \overline{\{0\}_U}$ is dense in the kernel of $V / \overline{\{0\}_V} \twoheadrightarrow W / \overline{\{0\}_W}$.
\end{lem}

\begin{proof}
It is direct to verify $0 \to U / \overline{\{0\}_U} \to V / \overline{\{0\}_V}$ is injective, $V / \overline{\{0\}_V} \to W / \overline{\{0\}_W} \to 0$ is surjective with quotient topology, and the sequence $0 \to \overline{\{0\}_U} \to \overline{\{0\}_V} \to \overline{\{0\}_W}$ is exact. 

Quotient maps of locally convex vector spaces are automatically open.
If $v \in V$ has image in $\overline{\{0\}_W}$, then for any open lattice $\Lambda$, $U \cap (v+\Lambda) \neq \emptyset$, implying $v \in \overline{U}$, which verifies our claim.
\end{proof}

\begin{prop}\label{HS comp exact}
If $0 \to U \to V \to W \to 0$ is a short exact sequence of locally convex spaces such that $U$ has subspace topology and $W$ has quotient topology, the sequence $0 \to \hat{U} \to \hat{V} \to \hat{W}$ is exact with strict morphisms.
\end{prop}
\begin{proof}
By Lem \ref{HS quot SES}, we may assume $U, V, W$ are Hausdorff. 
The sequence $0 \to U \to V \to W \to 0$ is exact at the first and last term with subspace and quotient topology, and the image of $U$ is dense in the kernel $Z$ of $V \twoheadrightarrow W$.

For any open lattice $\Lambda \subset V$, projection of $\Lambda$ is open as quotient map of locally convex spaces is open, we claim that $0 \to U / (\Lambda \cap U) \to V / \Lambda \to W / \Lambda \to 0$ is exact.
For any $v$, whose residue class is in the kernel of $V / \Lambda \twoheadrightarrow W / \Lambda$, $v \in Z+\Lambda = U + \Lambda$ as $U$ is dense in $Z$. 
The corresponding sequence of Hausdorff completion (\cite[\S 7]{Sch01}) $$0 \to \varprojlim\limits_{\Lambda \subset V ~ \mathrm{open}} U / (\Lambda \cap U) \to \varprojlim\limits_{\Lambda \subset V ~ \mathrm{open}} V / \Lambda \to \varprojlim\limits_{\Lambda \subset V ~ \mathrm{open}} W / \Lambda $$ is exact by \cite[Tag 02N1]{Sta}. 
The kernel between Hausdorff spaces is closed, and the strictness of maps can be directly checked as the set of open lattice $\{\Lambda \subset V ~ \mathrm{open}\}$ defines the topology of everything here.
\end{proof}

Let $K\{\{z,z^{-1}\}\}$ be the ring of entire functions of $\GG_m^{\mathrm{rig}}$. 
For a sequence of increasing radius $\{r_n\}$ greater than $1$ and $\lim\limits_n r_n = \infty$, the affinoid algebras associated to closed rigid annuli $A_n := \sO(\{r_n^{-1} \leq |z| \leq r_n\})$ define a Fr\'echet-Stein algebra structure of $K\{\{z,z^{-1}\}\} = \varprojlim_n A_n$ in the sense of \cite{ST03}.

\begin{prop}\label{FS flat}
Let $R$ be a Noetherian (left and right for the noncommutative case) ring, and $A = \varprojlim_n A_n$ be a Fr\'echet-Stein algebra, together with a compatible family $R \to A_n$ of flat maps.
Then $R \to A$ is flat.
\end{prop}
\begin{proof}
This is an abstraction of proof of \cite[Thm 4.11]{ST03}.
Let $J \subseteq R$ be a ideal.
As $R$ is Noetherian, $A \otimes_{R} J$ is finitely presented and therefore coadmissible by \cite[Cor 3.4]{ST03}.
By the definition of coadmissibility and \cite[Cor 3.1]{ST03},
\begin{eqnarray*} A \otimes_{R} J & \simeq & \varprojlim_n A_n \otimes_A (A \otimes_{R} J) \\ & \simeq & \varprojlim_n A_n \otimes_R J.
\end{eqnarray*}
As $R \to A_n$ is flat, $A_n \otimes_R J \to A_n$ is injective, we conclude that $A \otimes_R J \to A$ is injective since the projective limit is left exact.
\end{proof}

\begin{cor}\label{Gm flat}
The natural map $K[z] \to K\{\{z,z^{-1}\}\}$ is flat.
\end{cor}
\begin{proof}
We choose $r_n$ to be $p^n$.
Each map $O_K[z] \to \widehat{O_K[p^nz,p^nz^{-1}]}_p$ is a composition of localization and completion with respect to $p$. Base change to $K$, we have $K[z] \to \sO(\{r_n^{-1} \leq |z| \leq r_n\})$ is flat.
Now we can apply Prop \ref{FS flat}.
\end{proof}

\begin{thm}\label{FSEX}
Let $0 \to U \to V \to W \to 0$ be a short exact sequence in $\mathrm{Rep}^{\mathrm{top.c}}_K(Y^+)$ such that $U$ has subspace topology and $W$ has quotient topology.
Then $0 \to U_\fs \to V_\fs \to W_\fs \to 0$ is exact in $\mathrm{Rep}^\la_K(Y)$, and $U_\fs$ is a closed subspace and $W_\fs$ has quotient topology.
\end{thm}
\begin{proof}
We pick $t \in T$ satisfying \cite[Prop 4.1.6 (vi)]{Eme06A} as before.
The weight space associated to $\ZZ\cdot t$ in \cite[\S 6.4]{Eme17} is $\GG_m^{\mathrm{rig}}$.

By \cite[Prop 3.2.6]{Eme06A}, $U_\fs \hookrightarrow V_\fs$ is a closed embedding.
Locally convex topological $K$-vector spaces of compact type are reflexive. 
By Hahn-Banach (\cite[Cor 9.4]{Sch01}), it suffices to show that $$0 \to (W_\fs)'_b \to (V_\fs)'_b \to (U_\fs)'_b \to 0$$ is exact and has subspace and quotient topology.
By \cite[Prop 3.2.27]{Eme06A} and \cite[Lem 3.2.3]{Eme06A}, the sequence is the same as $$0 \to K\{\{t,t^{-1}\}\} \hat{\otimes}_{K[t]} W'_b \to K\{\{t,t^{-1}\}\} \hat{\otimes}_{K[t]} V'_b \to K\{\{t,t^{-1}\}\} \hat{\otimes}_{K[t]} U'_b \to 0,$$
which is the Hausdorff completion of $$0 \to K\{\{t,t^{-1}\}\} \otimes_{K[t]} W'_b \to K\{\{t,t^{-1}\}\} \otimes_{K[t]} V'_b \to K\{\{t,t^{-1}\}\} \otimes_{K[t]} U'_b \to 0.$$
It is exact by Cor \ref{Gm flat}, and equipped with subspace and quotient topology.
We conclude by applying Prop \ref{HS comp exact}.
\end{proof}

By \cite[Lem 3.1.1]{Eme06A}, $\Rep^{\mathrm{top.c}}_K(Y^+)$ is an additive category with images as closures of the usual images.
\begin{df}
If $C_\bullet$ is a chain complex of objects in $\Rep^{\mathrm{top.c}}_K(Y^+)$ with differential maps $d_i : C_i \to C_{i+1}$, then we define $i$-th \emph{Hausdorff cohomology} $H^i_h(C_\bullet)$ to be \[H^i_h(C_\bullet) := \ker(d_i) / \overline{\mathrm{im}(d_{i-1})}. \]
\end{df}
\begin{rem}
A morphism $C_\bullet \to D_\bullet$ induces a natural map $H^\ast_h(C_\bullet) \to H^\ast_h(D_\bullet)$ on Hausdorff cohomology, invariant under homotopy equivalence. 
\end{rem}

\begin{lem}\label{hat inj}
Let $0 \to U \to V$ be an injective map of Hausdorff locally convex spaces such that the induced map $\hat{U} \to \hat{V}$ on completions is strict, then $U$ has subspace topology in $V$ and $0 \to \hat{U} \to \hat{V}$ is injective.
\end{lem}
\begin{proof}
The strict map $\varprojlim\limits_{\Lambda \subset U ~ \mathrm{open}} U / \Lambda  \to \varprojlim\limits_{\Lambda' \subset V ~ \mathrm{open}} V / \Lambda'$ factors through \break
$\pi: \varprojlim\limits_{\Lambda \subset U ~ \mathrm{open}} U / \Lambda \to \varprojlim\limits_{\Lambda' \subset V ~ \mathrm{open}} U / (\Lambda' \cap U)$. 
For any open lattice $\hat{\Lambda}$ in $U$, let $\pi(\hat{\Lambda}) \subset \varprojlim\limits_{\Lambda' \subset V ~ \mathrm{open}} U / (\Lambda' \cap U)$ be the image of $\hat{\Lambda}$ under $\pi$. Since $\hat{U} \to \hat{V}$ is strict, $\pi(\hat{\Lambda})$ is open with respect to the subspace (of $\varprojlim\limits_{\Lambda' \subset V ~ \mathrm{open}} V / \Lambda'$) topology.  We deduce that $\pi(\hat{\Lambda}) \cap U = \hat{\Lambda}$ is open in $U$ with respect to the subspace topology of $\varprojlim\limits_{\Lambda' \subset V ~ \mathrm{open}} V / \Lambda'$ and $V$.
Therefore $U$ itself has subspace topology of $V$.
\end{proof}

Suppose we have a $G$-extension (with fixed representatives) $(C'_\bullet, i, p)$ of $C_\bullet \in \bC\bh^b(\contadg)$.
As in \S \ref{LR}, we have defined $J_B(C_\bullet)$ out of these data, with the Hecke action of $T^+$ on $H^\ast(J_B(C_\bullet))$, similarly there is a natural Hecke action of $T^+$ on Hausdorff cohomology $H^\ast_h \left((C_\bullet^{\la})^{N_0}\right)$. 
We may define the finite slope part of $H^\ast_h \left((C_\bullet^{\la})^{N_0}\right)_\fs$ using the $Y^+$ action via Def \ref{FSP}.

By Prop \ref{ES sketch}, $J_B(C_\bullet)$ is a chain complex of essentially admissible representations of $Y$.
By \cite[Prop 6.4.11]{Eme17}, the differentials $d^{\fs}_n : (V_n)_{\fs} \to (V_{n+1})_{\fs}$ are necessarily strict with closed images, the topology of $H^\ast(J_B(C_\bullet))$ are therefore canonically defined by subspace topology and quotient topology from $J_B(C_\bullet)$.
The $T^+$ actions on cohomology of the complex $J_B(C_\bullet)$ extends to essentially admissible $T$-representations by Lem \ref{sub ea}.

\begin{prop}\label{GAEA}
We have $H^\ast(J_B(C_\bullet)) \simeq H^\ast_h \left((C_\bullet^{\la})^{N_0} \right)_{\fs}$ as essentially admissible locally analytic $T$-representations, i.e., taking Hausdorff cohomology commutes with taking finite slope part. 
\end{prop}
\begin{proof}
We use $V_\bullet$ to denote $(C_\bullet^{\la})^{N_0}$ in the proof.

The Hausdorff cohomology do not depend on choices of $i$ and $p$, so we only have to establish the claimed isomorphism.
To apply Thm \ref{FSEX} to prove $H^\ast \left((V_\bullet)_{\fs}\right) \simeq H^\ast_h(V_\bullet)_{\fs}$, we have to prove (1): $\ker(d_n)_{\fs}$ agrees with $\ker(d^\fs_n)$, (2): image of $d^{\fs}_n$ coincides with $\overline{\mathrm{im}(d_n)}_{\fs}$, both as objects of $\mathrm{Rep}^\la_K(Y)$.
(1) holds by left exactness of finite slope part. For (2), $d_n$ factors through $\overline{\mathrm{im}(d_n)}$, it suffices to prove $(V_n)_{\fs} \to \overline{\mathrm{im}(d_n)}_{\fs}$ is an epimorphism in $\mathrm{Rep}^\la_K(Y)$.

It suffices to consider the dual side by \cite[Prop 1.2]{ST02J}, namely we want to prove that
if $\imath: U \to V$ is a strict map of spaces of compact type with closed image $W$, and if $\imath' : V’_b \to U’_b$ is injective, then $\imath$ is surjective.
By \cite[Prop 1.2]{ST02J}, $V’_b\to W’_b$ is a topological quotient with the topological embedding $W’_b \to U’_b$. By definition, $V’_b \to U’_b$ is strict. 
Consider the strict quotient $V \to V/\imath(U) \to 0$, $(V/\imath(U))’_b$ is in the kernel of $V’_b \to U’_b$. If $V’_b \to U’_b$ is injective, then $(V/\imath(U))’_b = 0$.
$V/\imath(U) = 0$ by the reflexive property.

By \cite[Prop 3.2.27]{Eme06A} and the same argument in Thm \ref{FSEX}, we want to prove the strict map \[ K\{\{t,t^{-1}\}\} \hat{\otimes}_{K[t]} \overline{\mathrm{im}(d_n)}'_b \to K\{\{t,t^{-1}\}\} \hat{\otimes}_{K[t]} (V_n)'_b \] is injective.
As $\overline{\mathrm{im}(d_n)}'_b \to (V_n)'_b$ is injective, $K\{\{t,t^{-1}\}\}$ is flat over $K[t]$, and the above map is strict again by \cite[Prop 6.4.11]{Eme17}, we can apply Lem \ref{hat inj} to deduce (2). Therefore we have proved the isomorphism $H^\ast \left((V_\bullet)_{\fs}\right) \simeq H^\ast_h(V_\bullet)_{\fs}$ in $\mathrm{Rep}^\la_K(Y)$.
As $H^\ast_h(V_\bullet)$ are $T^+$ continuous representations on spaces of compact type (objects of $\reptopt$ in \S \ref{LR}), $H^\ast_h(V_\bullet)_\fs$ are locally analytic $T$ representations since $T$ is generated by $T^+$ and $t^{-1}$.
We win by applying Lem \ref{sub ea}.
\end{proof}

\begin{cor}\label{fs indep}
If $f : C_\bullet \to D_\bullet \in \bK^b(\contadg)$ with a $G$-extension $f' : C'_\bullet \to D'_\bullet \in \bK^b(\mathrm{Ban}_G(K))$ in the sense of Def \ref{ext of mor}, with a diagram (abbreviation of all four diagrams in Def \ref{ext of mor})
$$\xymatrix{
C_\bullet \ar[r]^p \ar[d]^f & C'_\bullet \ar[d]^{f'} \ar[l]^i\\
D_\bullet \ar[r]^q & D'_\bullet \ar[l]^j,\\
}$$
 where $C_\bullet \stackrel[i]{p}{\leftrightarrows} C'_\bullet$ and $D_\bullet \stackrel[j]{q}{\leftrightarrows} D'_\bullet$ are $G$-extensions such that the diagram is commutative up to homotopy.
Then it induces a natural morphism on cohomology $H^\ast(J_B(C_\bullet)) \to H^\ast(J_B(D_\bullet))$ as essentially admissible locally analytic $T$-representations.
In particular, $H^\ast(J_B(C_\bullet))$ are independent of choices of $p : C'_\bullet \to C_\bullet$ with $i$ fixed or choices of $i : C_\bullet \to C'_\bullet$ with $p$ fixed of the extension and choices of $Y^+$.
\end{cor}
\begin{proof}
The diagram induces a natural $T^+$-equivariant morphism $H^\ast_h\left((C_\bullet^{\la})^{N_0}\right) \to H^\ast_h\left((D_\bullet^{\la})^{N_0}\right)$ in $\reptopt$.
We therefore have $H^\ast\left((C_\bullet^{\la})^{N_0}_\fs\right) \simeq H^\ast_h\left((C_\bullet^{\la})^{N_0}\right)_\fs$ as essentially admissible locally analytic $T$-representations by Prop \ref{GAEA}.
The Hausdorff cohomology $H^\ast_h\left((C_\bullet^{\la})^{N_0}\right)$ does not depend on choices of $p$ and $i$.
\end{proof}

\section{Comparisons of completed cohomology and overconvergent cohomology}\label{comp}
In this section we relate overconvergent cohomology considered by \cite{AS08}, \cite{Han17} and
the completed cohomology. 

Let $F$ be a number field. Let $\GG$ be a connected linear algebraic group over $F$.  
Let $S_p$ denote the set of places of $F$ dividing $p$, and let $\GG_p$ be $\prod_{v \in S_p} \GG(F_v)$. 
Let $E$ be a big enough extension of $\QQ_p$ which contains the images of all embeddings $F \hookrightarrow \Qpbar$.
For any good subgroup of $\GG(\AA^\infty_F)$ in the sense of \cite[\S 2.1]{ACC+18}, we write \begin{eqnarray}\label{locK} X_K:=X^{\GG}_K := \GG(F) \backslash X^{\GG} \times \GG(\AA^\infty_F)/ K. \end{eqnarray}
Fix a tame level $K^p=\prod\limits_{v \notin S_p \cup \infty} K_v$ ($K^p$ is a compact open subgroup of $\GG(\AA^{\infty,p}_F)$, $K_v$ is a compact open subgroup of $\GG(F_v)$), we use $K_pK^p$ to denote the good subgroup with tame level $K^p$ and level $K_p$ at $p$. 
For the rest of the section, all the level subgroups considered are assumed to be good.
Let $C^{\la}(K_p,E)$ be the space of locally analytic $E$-coefficient functions on $K_p$, hence a (left) admissible locally analytic representation of $K_p$.

By weak approximation, $\GG(F) \backslash \GG(\AA^\infty_F) / K$ is finite, 
\[ X_K = \coprod_{[x] \in \GG(F) \backslash \GG(\AA^\infty_F) /K} \Gamma^{\GG}_{x, K} \backslash X^{\GG}, \] 
\[ \overline{X}_K = \coprod_{[x] \in \GG(F) \backslash \GG(\AA^\infty_F) /K} \Gamma^{\GG}_{x, K} \backslash \overline{X}^{\GG}, \] where $\Gamma^{\GG}_{x, K}:=\GG(F) \cap xKx^{-1}$.

We fix a countable basis of open normal subgroups of $K_p$
$$K_p=K_p^0 \supset K_p^1 \supset \ldots \supset K_p^r \supset \ldots,$$
corresponding to a sequence of continuous covering maps of topological spaces
$$\cdots \rightarrow X_{r} \rightarrow \cdots \rightarrow X_{1} \rightarrow X_{0},$$ with $K_p$-equivariant actions on the right
and their Borel-Serre compactifications 
$$\cdots \rightarrow \overline{X}_{r} \rightarrow \cdots \rightarrow \overline{X}_{1} \rightarrow \overline{X}_{0},$$
where $X_r:=X_{K_p^r K^p}$ and each $X_r \hookrightarrow \overline{X}_r$ is a homotopy equivalence by \cite[\S 11]{BS73} for a good subgroup $K^0_p K^p$.

For $K=K_pK^p$, let $T_\bullet(X_0)$ be the set of singular simplices of Borel-Serre compactification $\overline{X}_0$ of $X_0:=X_K$. 
Fix a choice of finite triangulation of $\overline{X}_0$, we let $T^\circ_n(X_0)$ be the collection of $n$-dimensional singular simplices occuring in the finite triangulation, which induces a triangulation of the boundary $\partial \overline{X}_0$. 
Let $T_\bullet(X_r)$ (resp. $T^\circ_\bullet(X_r)$) be the set of simplices of $X_r$ (resp. the pullback of $T^\circ_\bullet(X_0)$ for the corresponding covering $\overline{X}_r \to \overline{X}_0$). 
If $\Delta \in T^\circ_n(X_0)$ for some $n$ then we let $T^\circ_n(X_r)_{/\Delta}$ denote the set of simplices in $T^\circ_n(X_r)$ lying over $\Delta$; the set $T^\circ_n(X_r)_{/\Delta}$ is then a principal homogeneous $K_p/K_p^r$-set acting on the right. 
We let $\hat{T}_{n/\Delta}$ denote the projective limit $\varprojlim\limits_{r} T^\circ_n(X_r)_{/\Delta}$; this is a profinite set which is principal homogeneous with respect to its natural $K_p$-action.
Let $T^\Int_\bullet(X_r)$ be the subset of $T^\circ_\bullet(X_r)$ whose interiors do not intersect with $\partial \overline{X}_0$, and $T^\partial_\bullet(X_r)$ be the complement of $T^\Int_\bullet(X_r)$ in $T^\circ_\bullet(X_r)$, i.e., $T^\circ_\bullet(X_r) = T^\Int_\bullet(X_r) \sqcup T^\partial_\bullet(X_r)$.
For each $\Delta \in T^\partial_\bullet(X_r)$, $\Delta$ has image in $\partial \overline{X}_0$.

On each finite level, the two complexes computing $\cO_E/p^s\cO_E$-coefficient cohomology via $T_\bullet(X_r)$ and $T^\circ_\bullet(X_r)$ are denoted as $A^\bullet_{r,s}$ and $S^\bullet_{r,s}$ whose $n$-th terms are
\begin{eqnarray*}
A^{n}_{r,s} & := & \prod_{\Delta^{\prime} \in T_{n}(X_r)} \Hom_\ZZ\left(\ZZ\Delta^{\prime}, \cO_E/p^s\cO_E \right) \\
& \stackrel{\sim}{\longrightarrow} & \prod_{\Delta \in T_{n}(X_{0})}\prod_{\Delta^{\prime} \in T_{n}(X_r)_{/ \Delta}} \Hom_\ZZ \left(\ZZ\Delta^\prime, \cO_E/p^s\cO_E \right). \\
S^{n}_{r,s} & := & \bigoplus_{\Delta^{\prime} \in T^\circ_n(X_r)} \Hom_\ZZ\left(\ZZ\Delta^{\prime}, \cO_E/p^s\cO_E \right) \\
& \stackrel{\sim}{\longrightarrow} & \bigoplus_{\Delta \in T^\circ_n(X_{0})}\bigoplus_{\Delta^{\prime} \in T^\circ_n(X_r)_{/ \Delta}} \Hom_\ZZ \left(\ZZ\Delta^\prime, \cO_E/p^s\cO_E \right). \\
\end{eqnarray*}
The natural map $A^n_{r,s} \to S^n_{r,s}$ is a quasi-isomorphism and admits a homotopy inverse.
We set $S^\bullet_s:=\varinjlim\limits_{r} S^\bullet_{r,s}$, $S^\bullet:=\varprojlim\limits_{s} S^\bullet_s$, $A^\bullet_s:=\varinjlim\limits_{r} A^\bullet_{r,s}$, $A^\bullet:=\varprojlim\limits_{s} A^\bullet_s$.
We endow discrete topology for $A^n_{r,s}$, $A^n_s$, $S^n_{r,s}$, $S^n_s$ and projective limit topology for $A^n$, $S^n$.
By \cite[Prop 1.2.12]{Eme06I}, $A^\bullet$ computes the completed cohomology.
A choice of compatible homotopy equivalences \begin{eqnarray}\label{Zequi} \bigoplus_{\Delta^{\prime} \in T^\circ_n(X_r)} \ZZ \Delta^{\prime} \leftrightarrows \bigoplus_{\Delta^{\prime} \in T_{n}(X_r)} \ZZ \Delta^{\prime} \end{eqnarray} induces $K_p$-equivariant homotopy equivalence $S^\bullet \leftrightarrows A^\bullet$.
We also have the following complex computing compactly supported completed cohomology
\begin{eqnarray*}
S^{n}_{c,r,s} & := & \bigoplus_{\Delta^{\prime} \in T^\Int_{n}(X_r)} \Hom_\ZZ\left(\ZZ\Delta^{\prime}, \cO_E/p^s\cO_E \right) \\
& \stackrel{\sim}{\longrightarrow} & \bigoplus_{\Delta \in T^\Int_{n}\left(X_{0}\right)}\bigoplus_{\Delta^{\prime} \in T^\Int_{n}(X_r)_{/ \Delta}} \Hom_\ZZ \left(\ZZ\Delta^\prime, \cO_E/p^s\cO_E \right),
\end{eqnarray*}
with a short exact sequence of chain complexes
\begin{eqnarray}\label{SES rs} 0 \to S^\bullet_{c,r,s} \to S^\bullet_{r,s} \to S^\bullet_{\partial,r,s} \to 0, \end{eqnarray}
where \begin{eqnarray*}
S^{n}_{\partial,r,s} & := & \bigoplus_{\Delta^{\prime} \in T^\partial_{n}(X_r)} \Hom_\ZZ\left(\ZZ\Delta^{\prime}, \cO_E/p^s\cO_E \right) \\
& \stackrel{\sim}{\longrightarrow} & \bigoplus_{\Delta \in T^\partial_{n}\left(X_{0}\right)}\bigoplus_{\Delta^{\prime} \in T^\partial_{n}(X_r)_{/ \Delta}} \Hom_\ZZ \left(\ZZ\Delta^\prime, \cO_E/p^s\cO_E \right).
\end{eqnarray*}
We set $S^\bullet_{c,s}:=\varinjlim\limits_{r} S^\bullet_{c,r,s}$, $S^\bullet_c:=\varprojlim\limits_{s} S^\bullet_{c,s}$, $S^\bullet_{\partial,s}:=\varinjlim\limits_{r} S^\bullet_{\partial,r,s}$, $S^\bullet_\partial:=\varprojlim\limits_{s} S^\bullet_{\partial,s}$.

The proof of \cite[Thm 2.1.5]{Eme06I} tells us that $S^\bullet \otimes_{\cO_E} E$ and $S^\bullet_{c} \otimes_{\cO_E} E$ are complexes of admissible continuous $E$-Banach representations of $K_p$ with the right action and
$$S^n \stackrel{\sim}{\longrightarrow} \prod_{\Delta \in T_{n}\left(X_{0}\right)} \mathcal{C}(\hat{T}_{n/\Delta}, \cO_E)$$
$$S_c^n \stackrel{\sim}{\longrightarrow} \bigoplus_{\Delta \in T^\Int_{n}\left(X_{0}\right)} \mathcal{C}(\hat{T}_{n/\Delta}, \cO_E),$$
where $\mathcal{C}(\hat{T}_{n/\Delta}, E)$ is the space of continuous functions on $\mathcal{C}(\hat{T}_{n/\Delta}, E)$ with the right regular $K_p$-actions: $K_p$ acts on $\hat{T}_{n/\Delta}$ on the right, inducing right actions on $\mathcal{C}(\hat{T}_{n/\Delta}, E)$.
Taking limits of the short exact sequences (\ref{SES rs}), we get a short exact sequence of chain complexes of admissible continuous $E$-Banach representations of $K_p$:
\begin{eqnarray}\label{SES ad} 0 \to S^\bullet_{c} \to S^\bullet \to S^\bullet_{\partial} \to 0. \end{eqnarray}
Similarly we describe natural right $\GG_p$-actions on $A^\bullet_s$ and $A^\bullet$ induced by the $\GG_p$-actions on $\varprojlim\limits_{r} T_n(X_r)$.

\begin{lem}\label{large Banach}
For each $n \geq 0$, $A^n$ is a $p$-adic separated, complete and torsion free $\cO_E$-module. 
Therefore $A^\bullet \otimes_{\cO_E} E$ is a complex of $E$-Banach space representation of $\GG_p$.
\end{lem}
\begin{proof}
It suffices to prove $A^n / p^s A^n \simeq A^n_s$ for any $s \geq 1$.
Consider the natural surjection $A^n / p^s A^n \twoheadrightarrow A^n_s$, we want to prove it is injective.
For $f \in A^n$ in the kernel, the image is the zero function on the inverse system $\varprojlim_r T_n(X_r)$ as it induces injective transition maps between $A^n_{r,s}$. 
And for any $s' \geq s$, the value of $f$ at $\Delta' \in T_n(X_r)$ divided by $p^s$ is well defined in $\cO_E / p^{s'} \cO_E$, giving rise to $f' \in A^n$ such that $f = p^s f'$.
It is also not hard to check $A^n$ is torsion free.
\end{proof}

If $\GG$ is a quasi-split connected reductive group over $F$ equipped with a Borel pair $(\BB,\TT)$, and $K_p = I_p$ is the Iwahori level with respect to the Borel pair at $p$-adic places, the group $I_p$ admits an Iwahori decomposition $I_p \simeq \overline{N}_p \times T_p \times N_p$ (with respect to the Borel pair).
$C^{\la}(I_p,E)$ has a closed subrepresentation
$$\Ind^{\la}_{I_p} := \{f \in C^{\la}(I_p,E) | f(gn)=f(g), ~ \mathrm{for} ~ \forall ~ n \in N_p\}.$$

When $K_p=I_p$, we use $S^\bullet_N$ to denote $\left((S^\bullet \otimes_{\cO_E} E)^{\la}\right)^{N_p}$, the $N_p$ invariant locally analytic vectors of $S^\bullet \otimes_{\cO_E} E$, and we similarly define $S^\bullet_{c,N}$.
On the other hand, we can form the complex $C^\bullet$ and $C^\bullet_c$ for computing the overconvergent cohomology and compactly supported overconvergent cohomology with coefficient $\Ind^{\la}_{I_p}$ (the term overconvergent cohomology was given in \cite[\S 3]{Han17}, but we will not use the exactly same definition here in this article).

For $[x] \in \GG(F) \backslash \GG(\AA^\infty_F) /I_pK^p$, we define the arithmetic group $\Gamma^{\GG}_{x,I_pK^p} := \GG(F) \cap xI_pK^px^{-1}$.
Each $\Delta \in T^\circ_n(X_0)$ is contained in a component of $X_0$, indexed by some $[x] \in \GG(F) \backslash \GG(\AA^\infty_F) /I_pK^p$, and $\Delta$ determines a rank one free $\ZZ[\Gamma^{\GG}_{x, I_pK^p}]$-module $F_\Delta$ formed by summing (taking $\ZZ$-linear combination) over all lifts of $\Delta$ to the universal cover of the component. And we define
$$C^\bullet := \bigoplus_{\Delta \in T^\circ_n(X_0)} \Hom_{\ZZ[\Gamma^{\GG}_{x, I_pK^p}]}(F_\Delta,\Ind^{\la}_{I_p})$$
$$C_c^\bullet := \bigoplus_{\Delta \in T^\Int_n(X_0)} \Hom_{\ZZ[\Gamma^{\GG}_{x, I_pK^p}]}(F_\Delta,\Ind^{\la}_{I_p}),$$
where the action of $\Gamma^{\GG}_{x,I_pK^p}$ on $\Ind^{\la}_{I_p}$ comes from the (left) action of $I_p$.
We remark that the cohomology of $C^\bullet$ may be considered as a variation of overconvergent cohomology defined in \cite[\S 3]{Han17}.

\begin{thm}\label{comp to overconvergent}
Using notation as above, we suppose $K_p = I_p$ as Iwahori level at $p$-adic places.
There is a canonical isomorphism of the two complexes
$$C^\bullet \stackrel{\sim}{\longrightarrow} S^\bullet_N, ~ C^\bullet_c \stackrel{\sim}{\longrightarrow} S^\bullet_{c,N}.$$
\end{thm}

\begin{proof}
To construct such a map is the same to assign $f_\Delta \in \Hom_{\ZZ[\Gamma^{\GG}_{x, I_pK^p}]}(F_\Delta,\Ind^{\la}_{I_p})$ an image in $\Ind^{\la}_{I_p}$ for any $\Delta \in T_n(X_0)$.
To do so, we choose a specific lift $\tilde{\Delta}$ to $X^\GG$ for each $\Delta$.
Consider the diagram (for any subgroup $K_p \subset I_p$),
$$\xymatrix{
& & \GG(F) \backslash \GG(F) x K_pK^p \times X^\GG / K_pK^p \ar@{->>}[dd]\\
X^\GG \ar@{->>}[rru] \ar@{->>}[rrd] & & \\
& & \GG(F) \backslash \GG(F) x I_pK^p \times X^\GG / I_pK^p\\
}$$
$\tilde{\Delta}$ determines a compatible sequence $\Delta_r \in T_n(X_r)$ lifting $\Delta$, hence an element $\tilde{\Delta} \in \hat{T}_{n/\Delta}$, $f_\Delta(\tilde{\Delta}) \in \Ind^{\la}_{I_p}$.
For any $\Delta_g \in \hat{T}_{n/\Delta}$, there exists $g \in I_p$ such that $\Delta_g = \tilde{\Delta}\cdot g$.
Let $f_\Delta$ map to a function $h_\Delta$
\begin{eqnarray*} 
h_\Delta : \hat{T}_{n/\Delta} & \to & E \\
h_\Delta(\Delta_g) & := & f_\Delta(\tilde{\Delta})(g),
\end{eqnarray*}
hence the image is in $\left((S^\bullet \otimes_{\cO_E} E)^{\la}\right)^{N_p}$.
If $\tilde{\Delta}'$ is another choice of lift of $\Delta$, there exists $\gamma \in \Gamma^{\GG}_{x, I_pK^p}$ such that $\tilde{\Delta}'=\gamma\cdot \tilde{\Delta}$.
By unraveling definition, $\tilde{\Delta}'=\tilde{\Delta}\cdot x^{-1}\gamma^{-1} x$,
\begin{eqnarray*} h_\Delta(\tilde{\Delta}) & = & f_\Delta(\tilde{\Delta}')(x^{-1}\gamma x) \\
& = & x^{-1}\gamma x f_\Delta(\tilde{\Delta})(x^{-1}\gamma x) \\
& = & f_\Delta(\tilde{\Delta})(1).
\end{eqnarray*}
Therefore our map is well defined and clearly an isomorphism with respect to a given degree.

For a $n$-dimensional simplice $\Delta$, we use $\partial_i \Delta$ ($0 \leq i \leq n$) to denote its ordered $(n-1)$-dimensional faces with $i$-vertex deleted.
For any $f \in \prod\limits_{\Delta \in T_{n-1}(X_0)} \Hom_{\ZZ[\Gamma^{\GG}_{x, I_pK^p}]}(F_\Delta,\Ind^{\la}_{I_p})$, its image via the differential map is $h=(h_\Delta)_{\Delta \in T_n(X_0)} \in S^n_N$ such that
$$h_\Delta(\tilde{\Delta}) = \sum_{i=0}^n (-1)^i f_{\partial_i \Delta}(\partial_i \tilde{\Delta})(1).$$
The identification commutes with differentials since if $\tilde{\Delta}=(\Delta_r) \in \hat{T}_{n/\Delta}$, then $\partial_i \tilde{\Delta}=(\partial_i \Delta_r) \in \hat{T}_{n-1/\Delta}$ for any $0 \leq i \leq n$.
\end{proof}

\begin{rem}\label{rem comp to overconvergent}
Using finiteness of direct sum, $C^\bullet_c$ computes compactly supported cohomology $H^\ast(X_{\GG,K}, \underline{\Ind^\la_{I_p}})$ of $X_{\GG,K}$ for $K = I_p K^p$ with coefficient local system arising from the representation $\Ind^{\la}_{I_p}$.
\end{rem}

We now illustrate how to relate the overconvergent cohomology in our sense to the usual cohomology with finite torsion coefficients.
We fix a countable basis of open normal subgroups of $I_p$
$$I_p=I_p^0 \supset I_p^1 \supset \ldots \supset I_p^r \supset \ldots.$$
And let $$\Ind^s_{I_p} := \{f \in C^{s,\an}(I_p,E) | f(gn)=f(g), ~ \mathrm{for} ~ \forall ~ n \in N_p\},$$ where $C^{s,\an}(I_p,E)$ consists of locally analytic functions on $I_p$ which restrict to analytic functions on $I_p^s$-cosets.
We have $\Ind^\la_{I_p} \simeq \varinjlim_{s} \Ind^s_{I_p}$, identifying $\Ind^s_{I_p}$ with $I_p^s$-analytic vectors of the locally analytic representation $\Ind^\la_{I_p}$.
If we want to prove vanishing of cohomology with coefficient $\Ind^\la_{I_p}$: we may reduce to prove vanishing of cohomology with coefficient $\Ind^s_{I_p}$, which has an integral structure as a representation of $I_p$.

\begin{lem}\label{loc m van}
Let $G$ be a group, $V$ be a $\QQ_p$-linear unitary Banach representation of $G$. Let $V^\circ$ be the unit ball in $V$. An algebra $T$ acts $\QQ_p$-linearly on $R\Gamma_G(V)$ in the derived category of $\QQ_p$ vector spaces. Let $\frak{m}$ be a prime ideal of $T$. Assume
\begin{itemize}
\item $V^\circ/p\cdot V^\circ \simeq \bigoplus\limits_{\ZZ} \FF_p$ ($\ZZ$ copies of $\FF_p$ as a direct sum) as $\FF_p$ vector space.
\item There exists a normal subgroup $N \trianglelefteq G$, such that $T$ acts compatibly on $R\Gamma_G(\FF_p)_{\frak{m}} \xrightarrow{\mathrm{Res}} R\Gamma_N(\FF_p)_{\frak{m}} \curvearrowleft G/N$ ($T$ acts on $R\Gamma_N(\FF_p)_{\frak{m}}$, $T$ action commutes with the $G/N$ action on $R\Gamma_N(\FF_p)_{\frak{m}}$, $\mathrm{Res}$ is $T$-equivariant) and $R^i\Gamma_N(\FF_p)_{\frak{m}}=0$ for $\forall i<d$. The $N$ action on $V^\circ/p\cdot V^\circ$ is trivial.
\end{itemize} 
Then we have $R^i\Gamma_G(V)_{\frak{m}}=0$ for $\forall i<d$. 

Furthermore, we assume additionally that there exists a connected manifold $M$ of dimension $N$ with homotopy type $K(G,1)$. We use $\underline{V}$ to denote the local system associated to $V$ on $M$. If $i > N-d$, we have $R^i\Gamma_c(M, \underline{V})_{\frak{m}}=0$.
\end{lem}
\begin{proof}
For $\forall i<d$, $R^i\Gamma_N(V^\circ / p \cdot V^\circ)_{\frak{m}}=0$ since $N$ acts trivially on $V^\circ / p \cdot V^\circ \simeq \bigoplus\limits_{\ZZ} \FF_p$ and direct sum commutes with cohomology and localization.  
By our assumptions and the Lyndon-Hochschild-Serre spectral sequence, $R^i\Gamma_G(\FF_p)_{\frak{m}}=0$ hence $R^i\Gamma_G(V^\circ/p^n \cdot V^\circ)_{\frak{m}}=0$ for any $n \in \ZZ_+$ and $\forall i<d$.
The inverse system $\{V^\circ/p^n \cdot V^\circ\}$ satisfies the Mittag-Leffler condition, therefore by \cite[Tag 02N1]{Sta}, $$R^i\Gamma_G(V^\circ)_{\frak{m}}=0, ~~ R^i\Gamma_G(V)_{\frak{m}}=0 ~ \mathrm{for} ~ \forall i<d.$$ 
By Poincar\'e duality, $R^i\Gamma_c(M, \FF_p)_{\frak{m}}=0$ for $\forall i > N-d$ hence $$R^i\Gamma_c(M, \bigoplus\limits_{\ZZ} \FF_p)_{\frak{m}}=0, ~ R^i\Gamma_c(M, V^\circ/p^n \cdot V^\circ)_{\frak{m}}=0 ~ \mathrm{for} ~ \forall i > N-d.$$
Applying the same Mittag-Leffler condition for the left exact functor $\Gamma_c$ yields the result $R^i\Gamma_c(M, \underline{V})_{\frak{m}}=0$ for $\forall i > N-d$.
\end{proof}

\begin{cor}\label{loc m van'}
Suppose $G = \Gamma^\GG_{x,I_pK^p}$, $M = \Gamma^\GG_{x,I_pK^p} \backslash X^\GG$ , $V = \Ind^s_{I_p}$ as a Banach representation of $I_p$, $T$ acts on $H^\ast_c(M, \underline{V})$ satisfying the conditions of Lem \ref{loc m van}.
The compactly supported cohomology $H^\ast_c(M, V)_\frakm$ of $M$ localized at $\frakm$ vanish above the degree $N-d$.
\end{cor}
\begin{proof}
The congruence subgroup $\Gamma^\GG_{x,I^n_pK^p} < G$ acts trivially on the reduction of $V^\circ$ if $n \gg s$.
\end{proof}

\section{Construction of eigenvarieties}\label{EV construction}
\begin{df}\label{EV data}
An eigenvariety data is a tuple $D = (T, M, A, \psi)$, where $T$ is a topologically finitely generated abelian locally $\QQ_p$-analytic group (\cite[Prop 6.4.1]{Eme17}), $M$ is an essentially admissible representation of $T$, $A$ is a commutative $\QQ_p$-algebra, and $\psi$ is a $\QQ_p$-algebra homomorphism $\psi : A \to \End_T(M)$.
\end{df}

In practice $T$ will be a maximal $p$-adic torus of our interested reductive group, $M$ will be a Jacquet module introduced in \S \ref{LR}, $A$ will be a Hecke algebra.
For the strong dual $M'_b$ of $M$, we define the $T$-action \[ t \cdot \lambda (m) := \lambda(t \cdot m) ~ \for ~ \forall \lambda \in M'_b, ~ \forall t \in T \] on $M'_b$ (it is direct to check that $t_1 t_2 \cdot \lambda = t_1 (t_2 \cdot \lambda)$ as $T$ is abelian).

\begin{prop}\label{EV machine}
Let $\widehat{T}$ be the character variety of $T$ parameterizing continuous characters of $T$.
Given an eigenvariety data $D = (T, M, A, \psi)$, there is a rigid analytic space $\mathscr{E} = \mathscr{E}(D)$ together with a locally finite morphism $\mathscr{E} \to \widehat{T}$, an algebra homomorphism $\psi_{\mathscr{E}} : A \to \sO(\mathscr{E})$.

A closed point $x \in \mathscr{E}(\Qpbar)$ gives rise to homomorphisms $\lambda : A \to \Qpbar$ and $\delta : T \to \Qpbar^\times$. And there is a natural bijection of sets (the natural map from left to right is constructed in the proof) $$\mathscr{E}(\Qpbar) \leftrightarrow \left\{\begin{array}{c}
(\lambda, \delta) \in \Hom(A, \Qpbar)  
\times \widehat{T}(\Qpbar) \text { such that } \\
\operatorname{Hom}_{T}(\mathds{1}(\delta), M[\lambda]) \neq 0
\end{array}\right\},$$
where $\mathds{1}(\delta)$ is the one-dimensional $\Qpbar$ representation of $T$ corresponding to $\delta$ and $M[\lambda] := M[\ker(\lambda)] \otimes_{A,\lambda} \Qpbar$ for the maximal ideal $\ker(\lambda)$.
\end{prop}
\begin{proof}
By \cite[Prop 2.3.2]{Eme06I}, $M'_b$ can be viewed as a rigid analytic coherent sheaf on $\widehat{T}$, together with the map $$\psi : A \to \End_{\sO(\widehat{T})}(M'_b).$$
Let $\cA$ be the sheaf of $\sO(\widehat{T})$-algebra generated by $\psi$.
We construct $\mathscr{E}$ to be the relative $\mathrm{Sp}$ of $\cA$, together with $\psi_{\mathscr{E}} : A \to \sO(\mathscr{E})$.
Local finiteness of $\mathscr{E} \to \widehat{T}$ follows from the fact that the homomorphism group between two finite modules is still finite.

Given a closed point $x \in \mathscr{E}(\Qpbar)$, let $\lambda_x := x \circ \psi_{\mathscr{E}}$, and $\delta_x : T \to \Qpbar^\times$ is given by the image of $\mathscr{E}(\Qpbar) \to \widehat{T}(\Qpbar)$.
Let $\mathcal{P}_x$ be the kernel of $\lambda_x$.
For $\delta \in \widehat{T}(\Qpbar)$, $M'_b |_{\delta}$, the fiber of $M'_b$ at $\delta$, is finite dimensional. 
By construction of $\mathscr{E}$, $x = (\lambda_x, \delta_x) \in \mathscr{E}(\Qpbar)$ implies $(M'_b |_{\delta_x}) / \mathcal{P}_x (M'_b |_{\delta_x}) \neq 0$ ($M'_b$ is a coherent sheaf of faithful $\cA$-modules).
Reversely, given $\lambda : A \to \Qpbar$ with kernel $\cP$ and $\delta \in \widehat{T}(\Qpbar)$, a maximal ideal $\frak{m}_x \subset A \otimes \sO(\widehat{T})$ is determined.
\[ (M'_b |_{\delta}) / \mathcal{P} (M'_b |_{\delta}) = M'_b / \frak{m}_x \cdot M'_b \neq 0 \] implies $\frak{m}_x$ lies above the annihilator of $M'_b$, meaning a corresponding point $x$ of $\frak{m}_x$ is on the eigenvariety $\mathscr{E}(\Qpbar)$.

By \cite[\S 6.4]{Eme17} and \cite[Prop 2.3.2]{Eme06I}, there is a antiequivalence of categories between the category of coherent rigid analytic sheaves on $\widehat{T}$, and the category of essentially admissible representations of $T$.
For an ideal $I \subset A$, we claim $M[I]'_b \simeq M'_b / I M'_b$.
Taking the dual of the monomorphism $M[I] \hookrightarrow M$, we have the natural map $M'_b / I M'_b \twoheadrightarrow M[I]'_b$ of $M'_b$ factoring through $M'_b / I M'_b$.
Similarly the dual of $M'_b / I M'_b$ admits a monomorphism to $M[I]$, yielding the reverse natural map $M[I]'_b \twoheadrightarrow M'_b / I M'_b$ by taking a double dual.
The condition that $(M'_b |_{\delta}) / \mathcal{P} (M'_b |_{\delta}) \neq 0$ is equivalent to the condition $\Hom_T(\mathds{1}(\delta), M[\lambda]) \neq 0$.
\end{proof}

We also illustrate how to construct morphisms out of ``sub'' or ``quotient'' eigenvariety data as follows.
\begin{prop}\label{EV func}
Given two eigenvariety data $D = (T, M, A, \psi)$ and $D' = (T, M', A', \psi')$ such that $M'$ is sub or quotient representation of $M$ and there is a $\QQ_p$-algebra homomorphism $f : A \to A'$ making $\psi$ and $\psi'$ compatible: the induced $A$ action via $\psi$ stabilizes $M'$ and it coincides with the action of $\psi' \circ f$.
Then there is a natural map of rigid spaces from $\mathscr{E}(D')$ to $\mathscr{E}(D)$.
If $f$ is surjective, then the natural map $\mathscr{E}(D') \to \mathscr{E}(D)$ is a closed embedding.
\end{prop}
\begin{proof}
$f$ induces a morphism from the sheaf of $\sO(\widehat{T})$-algebra for $D'$ to the sheaf of $\sO(\widehat{T})$-algebra for $D$.
If $f$ is surjective, then this map of sheaves of $\sO(\widehat{T})$-algebras is surjective.
Taking $\mathrm{Sp}$ of the morphism, we construct a natural map from $\mathscr{E}(D')$ to $\mathscr{E}(D)$.
\end{proof}

In the rest two sections, we adapt all notation and terminologies in \S 2 of \cite{ACC+18} and the previous section.
We define our derived eigenvarieties using results of previous sections as follows.

Let $S$ be a finite set of ``bad" finite places containing $S_p$ such that $K_v$ is hyperspecial for $v \notin S$.
Let $\cA_{K_p}$ be the abelian category of admissible $\cO_E[K_p]$-modules introduced in \cite[\S 1.2]{Eme06I} for the compact locally $\QQ_p$-analytic group $K_p$.
As in \S \ref{comp}, we define complexes $\pi(\GG, K^p, r, m)$, $\pi(\GG, K^p, m)$, $\tilde{\pi}(\GG, K^p, m)$ whose $n$-terms are (these are denoted as $S^\bullet_m$ and $A^\bullet_m$ in \S \ref{comp}) 
\[ \pi^n(\GG, K^p, r, m) := \bigoplus\limits_{\Delta^{\prime} \in T^\circ_{n}(X_r)} \Hom_\ZZ\left(\ZZ\Delta^{\prime}, \cO_E/p^m\cO_E \right), \]
$$\pi^n(\GG, K^p, m) := \varinjlim_{r}  \pi^n(\GG, K^p, r, m) ,$$
 $$\tilde{\pi}^n(\GG, K^p, m) := \varinjlim_{r} \prod\limits_{\Delta^{\prime} \in T_{n}(X_r)} \Hom_\ZZ\left(\ZZ\Delta^{\prime}, \cO_E/p^m\cO_E \right).$$
$\tilde{\pi}(\GG, K^p, m)$ is a well defined object in $\bK^b(\cO_E / p^m[\GG_p])$ and $$\pi^\circ(\GG, K^p) := \varprojlim\limits_{m} \pi(\GG, K^p, m) \in \bK^b(\cA_{K_p}),$$   $$\pi(\GG, K^p) := \pi^\circ(\GG, K^p) \otimes_{\cO_E} E \in \bK^b(\Ban^\ad_{K_p}(E)),$$
\[ \tilde{\pi}^\circ(\GG, K^p) := \varprojlim\limits_{m} \tilde{\pi}(\GG, K^p, m) \in \bK^b(\cO_E[\GG_p]), \]
$$\tilde{\pi}(\GG, K^p) := \tilde{\pi}^\circ(\GG, K^p) \otimes_{\cO_E} E \in \bK^b(\Ban_{\GG_p}(E)).$$

By \cite[Prop. 6.2.4]{Eme17}, the locally $\QQ_p$-analytic vectors $\pi(\GG, K^p)^{\la}$ is a complex of admissible locally $\QQ_p$-analytic representations of $K_p$ and their cohomology groups $\widetilde{H}^\ast(K^p)^{\la}$ are admissible locally $\QQ_p$-analytic representations of $\GG_p$.

We have an $K_p$-equivariant homotopy equivalence $\pi(\GG, K^p, m) \stackrel[i]{p}{\leftrightarrows} \tilde{\pi}(\GG, K^p, m)$, compatible with $m$, as discussed in section \S \ref{comp} (essentially coming from the homotopy equivalence of $\ZZ$-complexes (\ref{Zequi})).
We can take a projective limit (with respect to $m$) of such $K_p$-equivariant homotopy equivalences to get $K_p$-equivariant continuous mapping $i, p$ \[ \pi^\circ(\GG, K^p) \stackrel[i]{p}{\leftrightarrows} \tilde{\pi}^\circ(\GG, K^p), ~ \pi(\GG, K^p) \stackrel[i]{p}{\leftrightarrows} \tilde{\pi}(\GG, K^p). \]
In other words, $i, p$ are morphisms in $\bC\bh^b(\Ban_{K_p}(E))$.

\begin{lem}\label{Kp-equi}
The $K_p$-equivariant continous homotopy equivalence $\pi(\GG, K^p) \stackrel[i]{p}{\leftrightarrows} \tilde{\pi}(\GG, K^p)$ defines a $\GG_p$-extension $(\tilde{\pi}(\GG, K^p), i, p)$ of $\pi(\GG, K^p)$ in the sense of Def \ref{Gaction}.
\end{lem}
\begin{proof}
We have that $\tilde{\pi}(\GG, K^p)$ is a Banach representation of $\GG_p$ by Lem \ref{large Banach}.
\end{proof}

In the following we illustrate how to decompose $\pi^\circ(\GG, K^p)$ and $\pi(\GG, K^p)$ in terms of maximal ideals of derived Hecke algebra considered in \cite{GN16}.

For a compact open subgroup $K_p^r \subset K_p$, we form the following complexes with natural transitions in $r$ and $m$ 
\[ C(\GG, K^p, r, m) := \bigoplus\limits_{\Delta \in T^\circ_{\bullet}(X_0)} \bigoplus\limits_{\Delta^{\prime} \in T^\circ_{\bullet}(X_r)_{/\Delta}} (\cO_E/p^m\cO_E) \cdot \Delta^{\prime} \in \bK^b(\cO_E [[K_p]]), \]
\[ \widetilde{C}(\GG, K^p, r, m) := \bigoplus\limits_{\Delta^{\prime} \in T_{\bullet}(X_r)} (\cO_E/p^m\cO_E) \cdot \Delta^{\prime} \in \bK^b(\cO_E [[K_p]]). \]
The Iwasawa algebra $\cO_E [[K_p]]$-action on these complexes factors through the $\cO_E / p^m [K_p / K_p^r]$-action.
As discussed in \S \ref{comp}, there is a choice of compatible homotopy equivalences in $r$ and $m$
\[ C(\GG, K^p, r, m) \leftrightarrows \widetilde{C}(\GG, K^p, r, m). \]
We will sometimes regard $C(\GG, K^p, r, m)$ and $\widetilde{C}(\GG, K^p, r, m)$ as complexes or as the same object in $\bD(\cO_E[[K_p]])$ or $\bD(\cO_E / p^m [K_p / K_p^r])$.

For each $r$ and $m$, we may consider the derived Hecke algebra $\bT^S_\GG(K^p, r, m)$ as the image of $\cH(\GG(\AA^{\infty, S}_F), K^S) \otimes_{\ZZ} \cO_E \to \End_{\bD(\cO_E/p^m\cO_E[K_p / K_p^r])} (\widetilde{C}(\GG, K^p, r, m))$ (\cite[Def 2.1.11]{GN16}).
The (big) derived Hecke algebra \[ \bT^{S,\circ}_\GG := \varprojlim_{r,m} \bT^S_\GG(K^p, r, m) \] is semilocal with the decomposition \[ \bT^{S,\circ}_\GG = \bigoplus_{\frakm \in \mathrm{MaxSpec}(\bT^{S,\circ}_\GG)} \bT^{S,\circ}_{\GG, \frakm} \] by \cite[Lem 2.1.14]{GN16} (we note that although \cite{GN16} is set up for projective general linear groups, both of the proof of \cite[Lem 2.1.14]{GN16} and \cite[Lem 2.5]{KT17} hold true in general).

As complexes, let \[ C(\GG, K^p) := \varprojlim_m \varprojlim_r C(\GG, K^p, r, m), ~ \widetilde{C}(\GG, K^p) := \varprojlim_m \varprojlim_r \widetilde{C}(\GG, K^p, r, m), \]
\[ \mathrm{with} ~ C(\GG, K^p) \leftrightarrows \widetilde{C}(\GG, K^p). \]
It is key to notice that $C(\GG, K^p)$ is a perfect object of $\bD(\cO_E[[K_p]])$, each term of the perfect complex (representing) $C(\GG, K^p)$ is a finite free $\cO_E[[K_p]]$-module. 
We will sometimes regard $C(\GG, K^p)$ as a perfect complex or as an object of $\bD(\cO_E[[K_p]])$.

For each pair of $r$ and $m$, we have a decomposition (we identify $C(\GG, K^p, r, m)$ and $\widetilde{C}(\GG, K^p, r, m)$ in the derived category, and $C(\GG, K^p, r, m)_\frakm$ is defined to be an object in the derived category) \[ C(\GG, K^p, r, m) = \bigoplus_{\frakm \in \mathrm{MaxSpec}(\bT^{S,\circ}_\GG)} C(\GG, K^p, r, m)_\frakm \in \bD(\cO_E[[K_p]]), \] compatible in terms of $r$ and $m$.

By \cite[Lem 091D]{Sta}, we have \[ C(\GG, K^p) \simeq R \varprojlim_m \varprojlim_r C(\GG, K^p, r, m) \in \bD(\cO_E[[K_p]]). \]

The space of $\cO_E$-valued continuous functions on $K_p$ is dual to $\cO_E[[K_p]]$, we have the following (compatible) identifications as complexes of objects in $\cA_{K_p}$
\[ \pi(\GG, K^p, r, m) \simeq \Hom_{\cO_E} ( C(\GG, K^p, r, m), \cO_E / p^m \cO_E), \]
\[ \pi^\circ(\GG, K^p) \simeq \Hom_{\cO_E}^{\mathrm{cont}} ( C(\GG, K^p), \cO_E) \in \bC\bh(\cA_{K_p}). \]

As derived limits exist in $\bD(\cO_E[[K_p]])$ by \cite[Tag 0BK7]{Sta}, for each maximal ideal $\frakm$, \[ C(\GG, K^p)_\frakm := R \varprojlim_m \varprojlim_r C(\GG, K^p, r, m)_\frakm \in \bD(\cO_E[[K_p]]), ~ \mathrm{and} \] 
\begin{eqnarray}\label{Hecke decom} C(\GG, K^p) \simeq \bigoplus_{\frakm \in \mathrm{MaxSpec}(\bT^{S, \circ}_\GG)} C(\GG, K^p)_\frakm \in \bD(\cO_E[[K_p]]) \end{eqnarray}
with respect to the idempotent decomposition of the semilocal Hecke algebra $\bT^{S,\circ}_\GG$ as derived limits commute with finite direct sum. 
By \cite[Tag 08GA]{Sta}, $C(\GG, K^p)_\frakm$ is a perfect object in $\bD(\cO_E[[K_p]])$. And we let $C(\GG, K^p)_\frakm$ be represented by a perfect complex.
We will sometimes regard $C(\GG, K^p)_\frakm$ as a perfect complex or as an object of $\bD(\cO_E[[K_p]])$.

Choose a perfect complex $C(\GG, K^p)_\frakm$, we have the following definition of  \[ \pi^\circ(\GG, K^p)_\frakm := \Hom_{\cO_E}^{\mathrm{cont}} ( C(\GG, K^p)_\frakm, \cO_E) \in \bC\bh(\cA_{K_p}) \]
as a complex of admissible $\cO_E[K_p]$-modules with only finitely many nonzero terms. 
If we choose different perfect complex for $C(\GG, K^p)_\frakm$, the resulting $\pi^\circ(\GG, K^p)_\frakm$ would only differ by a homotopy equivalence by \cite[Lem 064B]{Sta}.
We may pass $\pi^\circ(\GG, K^p)_\frakm$ to $\bD(\cA_{K_p})$.

By dualizing the spherical Hecke action on $C(\GG, K^p, r, m)$, $C(\GG, K^p)$ and $C(\GG, K^p)_\frakm$, we have the Hecke actions
\begin{eqnarray} \bT^{S,\circ}_\GG \to \End_{\bD^b(\cO_E / p^m[K_p])}(\pi(\GG, K^p, m)),\end{eqnarray}
\begin{eqnarray}\label{derived spherical Hecke} \bT^{S,\circ}_\GG \to \End_{\bD(\cA_{K_p})}(\pi^\circ(\GG, K^p)), \end{eqnarray}
\begin{eqnarray}\label{rational derived spherical Hecke} \bT^{S,\circ}_\GG \otimes_{\cO_E} E \to \End_{\bD(\Ban^\ad_{K_p}(E))}(\pi(\GG, K^p))\end{eqnarray}
in various bounded derived categories of abelian categories of representations of $K_p$.
In particular, $\pi^\circ(\GG, K^p)_\frakm \in \bD(\cA_{K_p})$ is equipped with a Hecke action of $\bT^{S,\circ}_\GG$ factoring through $\bT^{S,\circ}_{\GG, \frakm}$ such that the decomposition (\ref{max ideals decomp}) is $\bT^{S,\circ}_\GG$-equivariant.
The previous discussions give rise to the following proposition.

\begin{prop}\label{Hecke semiloc}
For each maximal ideal $\frakm \subset \bT^{S,\circ}_\GG$,
\[ \pi(\GG, K^p)_\frakm := \pi^\circ(\GG, K^p)_\frakm \otimes_{\cO_E} E \] is a complex of injective admissible Banach representation of $K_p$.
Passing to $\bD^b(\Ban^\ad_{K_p}(E))$, it is equipped with the induced Hecke action of $\bT^S_\GG := \bT^{S,\circ}_\GG \otimes_{\cO_E} E$.

There is an isomorphism of objects in the derived category: \begin{eqnarray}\label{max ideals decomp} \pi(\GG, K^p) \simeq \bigoplus_{\frakm} \pi(\GG, K^p)_{\frakm} \in \bD^b(\Ban^\ad_{K_p}(E)), \end{eqnarray}
for $\frakm$ running through maximal spectrum of $\bT^{S,\circ}_\GG$.
\end{prop}

By the compatibility of $C(\GG, K^p, r, m)$, we will exhibit a ``uniform finiteness" for defining a complex representing $C(\GG, K^p, r, m)_\frakm$ (which does not depend on $r$, $m$) as follows.
\begin{lem}\label{uniform finiteness}
We have \[ C(\GG, K^p, r, m)_\frakm \simeq C(\GG, K^p)_\frakm \otimes_{\cO_E[[K_p]]}^{\bL} (\cO_E / p^m \cO_E [K_p / K_p^r]) \in \bD(\cO_E[[K_p]]). \]
\end{lem}
\begin{proof}
As complexes we have
\[ C(\GG, K^p, r, m) \simeq C(\GG, K^p) \otimes_{\cO_E[[K_p]]} (\cO_E / p^m \cO_E [K_p / K_p^r]). \]
We get the desired isomorphism by passing to the derived category and applying the idempotent decomposition (\ref{Hecke decom}).
\end{proof}
Therefore if $C(\GG, K^p)_\frakm$ is a perfect complex, we can use the complex \[ C(\GG, K^p)_\frakm \otimes_{\cO_E[[K_p]]} (\cO_E / p^m \cO_E [K_p / K_p^r])\] to represent $C(\GG, K^p, r, m)_\frakm$, it is perfect with respect to $(\cO_E / p^m \cO_E [K_p / K_p^r])$.
Furthermore, for each term of $C(\GG, K^p)_\frakm \otimes_{\cO_E[[K_p]]} (\cO_E / p^m \cO_E [K_p / K_p^r])$, the rank of such a module does not depend on $r$ and $m$.

We define 
\[ \pi(\GG, K^p, r, m)_\frakm \simeq R\Hom_{\cO_E[[K_p^r]]} ( C(\GG, K^p)_\frakm, \cO_E / p^m \cO_E ) \in \bD(\cO_E[K_p]). \]
\begin{lem}\label{limit loc}
Pick a maximal ideal $\frakm \subset \bT^{S,\circ}_\GG$. For each $r$, $m$, we can choose a complex representing $\pi(\GG, K^p, r, m)_\frakm \in \bC\bh(\cO_E / p^m \cO_E [K_p / K_p^r])$ compatible in both $r$ and $m$ such that $\varprojlim_m \varinjlim_r \pi(\GG, K^p, r, m)_\frakm \in \bC\bh(\Ban^\ad_{K_p}(E))$ represents $\pi^\circ(\GG, K^p)_\frakm$ when passing to $\bD^b(\Ban^\ad_{K_p}(E))$.
\end{lem}
\begin{proof}
Note that as complexes \[ \pi(\GG, K^p, r, m) \simeq \Hom_{\cO_E[[K_p^r]]} ( C(\GG, K^p), \cO_E / p^m \cO_E), \] and $\pi(\GG, K^p, r, m)$ as an object in $\bD(\cO_E[K_p])$ admits a $\bT^S_\GG(K^p, r, m)$ action (compatible in $r$, $m$).

We may represent $\pi(\GG, K^p, r, m)_\frakm$ by $\Hom_{\cO_E[[K_p^r]]} ( C(\GG, K^p)_\frakm, \cO_E / p^m \cO_E )$.
Then as an object of $\bC\bh(\cA_{K_p})$, $\pi^\circ(\GG, K^p)_\frakm$ can be chosen as \begin{eqnarray} \pi^\circ(\GG, K^p)_\frakm & \simeq &\varprojlim_m \Hom_{\cO_E}^{\mathrm{cont}} (C(\GG, K^p)_\frakm, \cO_E / p^m \cO_E ) \\ & \simeq & \varprojlim_m \varinjlim_r \Hom_{\cO_E[[K_p^r]]} ( C(\GG, K^p)_\frakm, \cO_E / p^m \cO_E) \\ & \simeq &  \varprojlim_m \varinjlim_r \pi(\GG, K^p, r, m)_\frakm, \end{eqnarray} regarded as chain complexes.
\end{proof}

Let $\BB_p:=\prod\limits_{v \in S_p} \BB(F_v), ~ \TT_p:=\prod\limits_{v \in S_p} \TT(F_v)$.
Let $\widehat{\TT}_p$ denote the rigid analytic variety over $\QQ_p$ that parametrizes the locally $\QQ_p$-analytic characters of $\TT_p$.
By Prop \ref{GAEA}, $H^\ast(J_{\BB_p}(\pi(\GG, K^p)))$ are essentially admissible representations of $\TT_p$.
Recall $\bT^S_\GG = \bT^{S,\circ}_\GG \otimes_{\cO_E} E$ in Prop \ref{Hecke semiloc}, we have the Hecke actions
$$\psi: \bT^S_\GG \to \End_{\sO(\widehat{\TT}_p)} (H^\ast(J_{\BB_p}(\pi(\GG, K^p)))'_b),$$ 
$$\psi_n: \bT^S_\GG \to \End_{\sO(\widehat{\TT}_p)} (H^n(J_{\BB_p}(\pi(\GG, K^p)))'_b).$$

By Prop \ref{Hecke semiloc}, $\pi(\GG, K^p)_{\frakm}$ is a complex of injective objects for each $\frakm$.
The complex $\left(\pi(\GG, K^p)^{\la}\right)^{N_p}$ is homotopic equivalent to $\bigoplus_{\frakm} \left(\pi(\GG, K^p)_{\frakm}^{\la}\right)^{N_p}$ (as $N_p$ in \S \ref{comp}).

The cohomology $H^\ast\left((\pi(\GG, K^p)_{\frakm}^{\la})^{N_p}\right)$ decomposes as
\[ H^\ast\left((\pi(\GG, K^p)^{\la})^{N_p}\right) \simeq \bigoplus_{\frakm} H^\ast\left((\pi(\GG, K^p)_{\frakm}^{\la})^{N_p}\right). \]
Hecke operators at $p$ commute with the spherical Hecke action, therefore preserving each $H^\ast_h\left((\pi(\GG, K^p)_{\frakm}^{\la})^{N_p}\right)$-component.
According to Prop \ref{GAEA}, we may define $H^\ast\left(J_{\BB_p}(\pi(\GG, K^p))\right)_\m$ to be the finite slope part of cohomology $H^\ast_h\left((\pi(\GG, K^p)_{\frakm}^{\la})^{N_p}\right)$ of the complex $\left(\pi(\GG, K^p)_{\frakm}^{\la}\right)^{N_p}$.

As direct summands, $H^\ast\left(J_{\BB_p}(\pi(\GG, K^p))\right)_\frakm$ are essentially admissible as well.
Similarly, we have 
$$\psi_\frakm: \bT^S_\GG \to \End_{\sO(\widehat{\TT}_p)} \left((H^\ast\left(J_{\BB_p}(\pi(\GG, K^p))\right)_\m)'_b\right),$$
$$\psi_{n,\m}: \bT^S_\GG \to \End_{\sO(\widehat{\TT}_p)} \left((H^n\left(J_{\BB_p}(\pi(\GG, K^p))\right)_\m)'_b\right).$$

Let $$D = (\TT_p, H^\ast\left(J_{\BB_p}(\pi(\GG, K^p))\right), \bT^S_\GG, \psi),$$  $$D_n = (\TT_p, H^n\left(J_{\BB_p}(\pi(\GG, K^p))\right), \bT^S_\GG, \psi_n),$$ $$D_\frakm = (\TT_p, H^\ast\left(J_{\BB_p}(\pi(\GG, K^p))\right)_\frakm, \bT^S_\GG, \psi_\frakm),$$ $$D_{n,\frakm} = (\TT_p, H^n\left(J_{\BB_p}(\pi(\GG, K^p))\right)_\frakm, \bT^S_\GG, \psi_{n,\m})$$ be the eigenvariety data defined by Def \ref{EV data}. 

\begin{df}\label{EVs}
We define the eigenvarieties $\cE_\GG(K^p), ~\cE^n_\GG(K^p)$, $\cE_\GG(K^p,\frakm)$, $\cE^n_\GG(K^p,\frakm)$ to be the eigenvarieties associated to $D$, $D_n$, $D_\frakm$, $D_{n,\m}$ by Prop \ref{EV machine}.
\end{df}

\section{An application to eigenvariety of quasi-split unitary group}\label{App uni}
Let $F$ be an (imaginary) CM number field with maximal totally real subfield $F^+$.
Let $E$ be a big enough extension of $\QQ_p$ which contains the images of all embeddings $F \hookrightarrow \Qpbar$.
Let $\widetilde{G}_{/\cO_{F^+}}$ be the quasi-split unitary group with respect to the extension $F/F^+$ and the Hermitian form 
\begin{eqnarray*}
J_n = \begin{pmatrix}
0 & \Psi_n \\
-\Psi_n & 0
\end{pmatrix}, 
\end{eqnarray*} where $\Psi_n$ is the matrix with $1$’s on the anti-diagonal and $0$’s elsewhere.
We choose a standard Siegel parabolic subgroup $P$ containing upper triangular Borel pair $(B, T)$ and maximal $F^+$-split torus $S \subseteq T$ of $\widetilde{G}$.
Similarly let $T \subset G \subset P \subset \widetilde{G}$ be the standard Levi subgroup of $P$. 
We write $P=U \rtimes G$ for the unipotent radical $U$ of $P$, $B = N \rtimes T$ for the unipotent radical $N$ of $B$.
Here we have $G \simeq \Res_{\cO_F/\cO_{F^+}} \GL_n$.

We write $\widetilde{X} = X^{\widetilde{G}}$ and $X = X^G$. Similarly, we will use the symbols $\widetilde{K}$ and $K$ to denote good subgroups of $\widetilde{G}(\A_{F^+}^\infty)$ and $G(\A_{F^+}^\infty)$, respectively.

Let $\wtI_p$ and $I_p$ be the Iwahori subgroups of $\wtG_p$ and $G_p$, compatible with respect to the Borel pair.
Let $P_p := P(F^+_p)$.
Let $\wtG_0 = \wtG(\cO_{F^+_p})$ (resp. $G_0 = G(\cO_{F^+_p})$) be a hyperspecial maximal compact subgroup of $\wtG_p$ (resp. $G_p$).
Let $P_0 := P(\cO_{F^+_p})$, $N_0 := \wtG_0 \cap N(F^+_p)$.
Let $\widetilde{K} \subset \widetilde{G}(\A_{F^+}^\infty)$ be a good subgroup which is decomposed with respect to $P$ (\cite[\S 2.1.2]{ACC+18}), and $\widetilde{K}_P:=\wtK \cap P(\A_{F^+}^\infty) = \widetilde{K}_U \rtimes K$.
For any finite place $v$ of $F^+$, we use $\widetilde{K}_v$ (resp. $K_v$) to denote the level $\widetilde{K}$ at $v$ (level $K$ at $v$).
Let $S$ be a finite set of bad finite places containing $S_p$ such that $\widetilde{K}_v$ (resp. $K_v$) is hyperspecial for $v \in S$.
We let $\bT^S:=\bT^S_G$ and $\widetilde{\bT}^S:=\bT^S_{\wtG}$.
There is an unramified Satake transform $\cS: \widetilde{\bT}^S \to \bT^S$ defined in \cite[\S 2.2.6]{NT15} (see \cite[\S 5.1]{NT15} for more explicit descriptions).

With notation in \S \ref{EV construction}, we define 
\[ \pi(K^p) := \pi(G, K^p) \in \bD^b(\Ban^\ad_{G_0}(E)), ~~ \Pi(\wtK^p) := \pi(\wtG, \wtK^p) \in \bD^b(\Ban^\ad_{\wtG_0}(E)) \] equipped with 
\begin{eqnarray}\label{Hecke G tildeG} \bT^S \to \End_{\bD^b(\Ban^\ad_{G_0}(E))}(\pi(K^p)), ~~ \widetilde{\bT}^S \to \End_{\bD^b(\Ban^\ad_{\wtG_0}(E))}(\Pi(\wtK^p)). \end{eqnarray}

Recall $\partial \mathfrak{X}_{\widetilde{G}}$ as in \S \ref{notation}. 
We define $\Pi_\partial(\wtK^p) \in \bK^b(\Ban^\ad_{\wtG_0}(E))$ to be a chosen complex of admissible continuous Banach representations of $\wtG_0$ arising from a choice of finite triangulation of Borel-Serre boundary of locally symmetric space of $\wtG$ with level $\wtG_0 \cdot \wtK^p$ (similar to $S^\bullet$ as in \S \ref{comp}). 
We define $\widetilde{\widetilde{\pi}}_\partial(\wtK^p) \in \bK^b(\mathrm{Ban}_{\wtG_p}(E))$ to be the complex of continuous Banach representations of $P_p$ arising from all possible singular simplices (similar to $A^\bullet$ as in \S \ref{comp}).

They are equipped with Hecke actions:
\[ \widetilde{\bT}^S \to \End_{\bD^b(\Ban^\ad_{\wtG_0}(E))}(\Pi_\partial(\wtK^p)), ~~ \widetilde{\bT}^S \to \End_{\bK^b(\Ban_{\wtG_p}(E))}(\widetilde{\widetilde{\pi}}_\partial(\wtK^p)). \]

If $\frakm \subset \bT^S$ is a maximal ideal, let $\widetilde{\frakm} = \cS^\ast(\m) \subset \widetilde{\bT}^S$ be the pullback of $\frakm$ in $\widetilde{\bT}^S$. 
As $\widetilde{\bT}^S / \widetilde{\frakm} \hookrightarrow \bT^S / \frakm$ as a subring of a finite field (therefore a subfield), $\widetilde{\frakm}$ is a maximal ideal of $\widetilde{\bT}^S$.

By the same reason before Prop \ref{Hecke semiloc} in \S \ref{EV construction}, we can decompose $\pi(K^p)$, $\Pi(\wtK^p)$, $\Pi_\partial(\widetilde{K}^p)$ and \[ \pi(K^p)_\frakm \in \bD^b(\Ban^\ad_{G_0}(E)), ~~ \Pi(\wtK^p)_\frakm, ~ \Pi_\partial(\widetilde{K}^p)_{\widetilde{\m}} \in \bD^b(\Ban^\ad_{\wtG_0}(E)) \] are all well defined for each relevant maximal ideal of the corresponding derived Hecke algebras. 

\begin{prop}\label{boundary direct summand}
Suppose $\m$ is a non-Eisenstein ideal (\cite[Def 2.3.6]{ACC+18}).
Let $\widetilde{K} \subset \widetilde{G}(\A_{F^+}^\infty)$ be a good subgroup which is decomposed with respect to $P$.
We have $\Ind_{P_0}^{\wtG_0} \pi(K^p)_{\frakm}$ is a $\widetilde{\bT}^S$-equivariant direct summand of $\Pi_\partial(\widetilde{K}^p)_{\widetilde{\frakm}}$ in $\bD^b(\Ban^\ad_{\wtG_0}(E))$.
\end{prop}

Here and below we will write Inf($C_\bullet$) when viewing a complex $C_\bullet$ of representations of $G_0$ as a complex of representations of $P_0$ via $P_0 \to G_0$. 
We also write $\pi(K^p)_{\m}$ as the inflation complex $\mathrm{Inf} (\pi(K^p)_{\m})$ of representations of $P_0$ whose underlying complex is $\pi(K^p)_{\m}$ in order to lighten the notation.

\begin{proof}
The claim follows from a combination of proof of \cite[Thm 5.4.1]{ACC+18}, uniform finiteness of Lem \ref{uniform finiteness}, Lem \ref{limit loc}, together with Prop \ref{ind projlim}. 
\end{proof}

Let \[ \pi_P(\widetilde{K}^p) := \pi(P, \wtK^p \cap P(\AA^{\infty,p}_{F^+})) \in \bK^b(\Ban^\ad_{P_0}(E)), \]
\[ \mathrm{and} ~ \widetilde{\pi}_P(\widetilde{K}^p) := \widetilde{\pi}(P, \wtK^p \cap P(\AA^{\infty,p}_{F^+})) \in \bK^b(\mathrm{Ban}_{P_p}(E)) \]
as in \S \ref{EV construction}. 
There is a $P_p$-extension \begin{eqnarray}\label{Pp ext} \pi_P(\widetilde{K}^p) \stackrel[i]{p}{\leftrightarrows} \widetilde{\pi}_P(\widetilde{K}^p)\end{eqnarray} constructed in \S \ref{comp} by Lem \ref{Kp-equi}.

There are corresponding Hecke actions
\[ \psi_\partial : \widetilde{\bT}^S \to \End_{\bD^b(\Ban^\ad_{P_0}(E))}(\pi_P(\widetilde{K}^p)), \]
\[ \psi_\partial : \widetilde{\bT}^S \to \End_{\bK^b(\Ban_{P_p}(E))}(\widetilde{\pi}_P(\widetilde{K}^p)) \]
and decomposition of $\pi_P(\wtK^p)$ into components indexed by maximal ideals similarly to (\ref{Hecke G tildeG}).

For every finite level $\widetilde{K}_P=\widetilde{K}_U \rtimes K$, we have a torus bundle $X^P_{\wtK_P} \twoheadrightarrow X^G_K$.
And taking levels at $p$ as open compact subgroups of $G_0$, we get a $\widetilde{\bT}^S$-equivariant natural map (similar to \cite[(5.4.2)]{ACC+18}) \[ \mathrm{Inf} (\pi(K^p)) \to \pi_P(\wtK^p) \in \bD^b(\Ban^\ad_{P_0}(E)), \] where the Hecke action on the left hand side factors through the Satake morphism $\cS$.
By decomposing with respect to $\widetilde{\frakm}$, we have $\mathrm{Inf} (\pi(K^p)_\frakm) \to \pi_P(\wtK^p)_{\widetilde{\frakm}}$.
They are quasi-isomorphisms since completed cohomology for $U$ (whose locally symmetric space is torus) concentrates in degree $0$.

We may represent $\mathrm{Inf} (\pi(K^p)_\frakm)$ or $\pi_P(\wtK^p)_{\widetilde{\m}}$ by an injective resolution of objects in $\Ban^\ad_{P_0}(E)$, so does $\Ind^{\wtG_0}_{P_0} \pi(K^p)_\frakm$ as a complex of injective objects in $\Ban^\ad_{\wtG_0}(E)$ by Lem \ref{Ind Inj}.

There is a $\wtG_p$-extension \begin{eqnarray}\label{Ind Pp ext} \Ind^{\wtG_0}_{P_0} \pi_P(\wtK^p) \stackrel[\tilde{i}]{\tilde{p}}{\leftrightarrows} \Ind^{\wtG_p}_{P_p} \widetilde{\pi}_P(\wtK^p)\end{eqnarray} by Prop \ref{Ind ext} induced from (\ref{Pp ext}).
Taking locally analytic vectors, (\ref{Ind Pp ext}) induces a $\wtG_0$-equivariant homotopy equivalence $$\operatorname{la-Ind}^{\wtG_0}_{P_0} ( \pi_P(\wtK^p)^\la) \stackrel[\tilde{i}]{\tilde{p}}{\leftrightarrows} \operatorname{la-Ind}^{\wtG}_P (\widetilde{\pi}_P(\wtK^p)^\la)$$ by Thm \ref{Ind la}.
We can form essentially admissible representations \[ H^\ast\left(J_{B_p}(\Ind^{\wtG_0}_{P_0} \pi_P(\wtK^p) )\right)_{\widetilde{\m}} \] similar to the discussion before Def \ref{EVs}.

\begin{df}
We call the \emph{induction eigenvariety $\mathscr{E}^d(\Ind, G, K^p, \m)$ for $G \simeq \Res_{\cO_F/\cO_{F^+}} \GL_n$ of degree $n$, type $\m$} for the eigenvariety associated to the eigenvariety data (recall $\widetilde{\m}$ is the pullback of $\m$) $$D^\Ind_{G,d} := (T_p, H^d\left(J_{B_p} (\Ind^{\wtG_0}_{P_0} \pi_P(\wtK^p) )\right)_{\widetilde{\m}}, \widetilde{\bT}^S, \psi_\partial).$$
\end{df}

This induction eigenvariety $\mathscr{E}^d(\Ind, G, K^p, \m)$ should be closely related to eigenvariety for $G$. 
But we do not know the exact relationship.

Let $\mathscr{E}^d_{\wtG}(\wtK^p, \widetilde{\frakm})$ be the eigenvariety for $\wtG$ and $\widetilde{\frakm}$ constructed in \S \ref{EV construction}.
\begin{thm}\label{main glo}
Suppose $F^+ \neq \QQ$ and $\widetilde{\m}$ satisfies the conditions of \cite[Thm 1.1]{CS19}.
If $d = [F^+ : \QQ]n^2$ is the middle dimension of $X^{\wtG}$, there is a closed embbeding from $\mathscr{E}^d(\Ind, G, K^p, \m)$ to the $\wtG$-eigenvariety $\mathscr{E}^d_{\wtG}(\wtK^p, \widetilde{\frakm})$.
\end{thm}
\begin{proof}
First, we exhibit $H^d\left(J_{B_p} (\Ind^{\wtG_0}_{P_0} \pi_P(\wtK^p) )\right)_{\widetilde{\m}}$ as a quotient of \break $H^d\left(J_{B_p} (\Pi(\wtK^p) ) \right)_{\widetilde{\frakm}}$.
By exactness of finite slope part of Thm \ref{FSEX} and Prop \ref{boundary direct summand}, it suffices to prove the surjectivity of $H^d\left((\Pi(\wtK^p)^\la)^{N_0}\right)_{\widetilde{\frakm}} \to H^d \left((\Pi_\partial(\widetilde{K}^p)^\la)^{N_0}\right)_{\widetilde{\frakm}}$ (and therefore the surjectivity of corresponding Hausdorff cohomology).

Consider the short exact sequence (\ref{SES ad}) of admissible continuous $E$-Banach representations of $\wtI_p$ for $\wtG$ with tame level $\wtK^p$ of \S \ref{comp}:
$$0 \to S^\bullet_c(\wtG, \wtK^p) \to S^\bullet(\wtG, \wtK^p) \to S^\bullet_\partial(\wtG, \wtK^p) \to 0,$$
where $S^\bullet(\wtG, \wtK^p)$ is homotopy equivalent to $\Pi(\wtK^p)$ and $S^\bullet_\partial(\wtG, \wtK^p)$ is homotopy equivalent to $\Pi_\partial(\widetilde{K}^p)$.
By the long exact sequence of $N_0$ cohomology associated to this short exact sequence and Thm \ref{comp to overconvergent}, it suffices to prove vanishing of \break $H^{d+1}_c (X_{\wtG, \wtI_p\wtK^p}, \underline{\Ind^\la_{\wtI_p}})_{\widetilde{\frakm}}$.
This follows from the main result of \cite[Thm 1.1]{CS19} and corollary \ref{loc m van'} in \S \ref{comp}.
We win by applying Prop \ref{EV func} to the quotient essentially admissible representation $H^d \left(J_{B_p}(\Ind^{\wtG_0}_{P_0} \pi_P(\wtK^p) ) \right)_{\widetilde{\m}}$ of $H^d \left(J_{B_p}(\Pi(\wtK^p)) \right)_{\widetilde{\frakm}}$.
\end{proof}

\end{document}